\documentclass[11pt]{amsart}
\usepackage[numbers, square]{natbib}
\usepackage{amssymb}
\usepackage{amsmath}
\usepackage{amsfonts}
\usepackage{geometry}
\usepackage{graphicx}
\usepackage{amsthm}
\usepackage{hyperref, esint}
\usepackage[dvipsnames]{xcolor}

\providecommand{\U}[1]{\protect\rule{.1in}{.1in}}
\theoremstyle{plain}

\newtheorem{cor}{Corollary}[section]

\newtheorem{lemma}{Lemma}[section]

\newtheorem{prop}{Proposition}[section]
\newtheorem{remark}{Remark}

\newtheorem{theorem}{Theorem}
\numberwithin{equation}{section}

\newcommand{\disp}{\displaystyle}

\DeclareMathOperator{\osc}{osc}

\DeclareMathOperator{\di}{div}

\newcommand{\eps}{\varepsilon}

\newcommand{\al}{\alpha}
\newcommand{\be}{\beta}
\newcommand{\ga}{\gamma}
\newcommand{\de}{\delta}

\newcommand{\la}{\lambda}

\newcommand{\om}{\omega}
\newcommand{\Om}{\Omega}
\newcommand{\si}{\sigma}

\newcommand{\iny}{\infty}
\newcommand{\del}{ \partial}
\newcommand{\su}{\subset}
\newcommand{\LP}{\Delta}
\newcommand{\gr}{\nabla}

\newcommand{\norm}[1]{\left\| #1\right\|}
\newcommand{\innp}[1]{\left< #1 \right>}
\newcommand{\abs}[1]{\left\vert#1\right\vert}
\newcommand{\set}[1]{\left\{#1\right\}}
\newcommand{\brac}[1]{\left[#1\right]}
\newcommand{\pr}[1]{\left( #1 \right) }

\newcommand{\N}{\ensuremath{\mathbb{N}}}

\newcommand{\R}{\ensuremath{\mathbb{R}}}

\newcommand{\C}{\ensuremath{\mathbb{C}}}

\title[Frequency function approach]{A frequency function approach to quantitative unique continuation for elliptic equations}

\author{Blair Davey}
\address[B. Davey]{Department of Mathematical Sciences, Montana State University, Bozeman, MT 59717}
\email{\textcolor{blue}{\href{mailto:}{blairdavey@montana.edu}}}
\thanks{B. D. is supported in part by the NSF CAREER DMS-2236491.}

\date{}

\begin{document}

\begin{abstract}
We investigate the quantitative unique continuation properties of solutions to second-order elliptic equations with lower-order terms.
In particular, we establish quantitative forms of the strong unique continuation property for solutions to generalized Schr\"odinger equations of the form $- \di\pr{A \gr u} + W \cdot \gr u + V u = 0$, where we assume that $A$ is bounded, elliptic, symmetric, and Lipschitz continuous, while $W$ belongs to $L^\iny$ and $V$ belongs to $L^p$.
We also study the global unique continuation properties of solutions to these equations, establishing results that are related to Landis' conjecture concerning the optimal rate of decay at infinity.
Versions of the theorems in this article have been previously proved using Carleman estimates, but here we present novel proof techniques that rely on frequency functions.
\end{abstract}

\maketitle

\section{Introduction}

A partial differential operator $L$ is said to have the {\em unique continuation property} (UCP) if whenever $u$ is a solution to $L u = 0$ in some domain $\Om$ and there exists an open set $U \su \Om$ for which $u \equiv 0$ in $U$, then it necessarily follows that $u \equiv 0$ in $\Om$.
We say that $L$ has the {\em strong unique continuation property} (SUCP) if whenever $u$ is a solution to $L u = 0$ in $\Om$ and there exists a point $x_0 \in \Om$ at which $u$ vanishes to infinite order, then $u \equiv 0$ in $\Om$.
The study of (strong) unique continuation for operators of the form {$- \di\pr{A \gr } + W \cdot \gr + V$} has a long history, see for example \cite{Car39}, \cite{Aro57}, \cite{AKS62}, \cite{SS80}, \cite{ABG81}, \cite{JK85}{, and \cite{KT01}.
The results of \cite{AKS62} show that the SUCP holds for differential inequalities associated to operators with $A \in C^{0,1}$ and bounded lower-order terms.}
Jerison and Kenig prove in \cite{JK85} that the SUCP holds for $-\LP + V$ in $\R^n$ for $n \ge 3$ if $V \in L^{n/2}$.
%For generalized Schr\"odinger operators of the form \eqref{ellipticOp}, {see, and \cite{KT01}, for example.}
The results of \cite{KT01} show that the SUCP holds whenever the leading coefficients are Lipschitz continuous and the lower order terms satisfy very general conditions.
{For example, \cite{KT01} establishes the SUCP when $\abs{x} \abs{\gr A} \in \ell^1(L^\iny)$, $V \in c_0(L^{\frac n 2})$, and $W \in \ell^1(L^n)$.}
For $n \ge 3$, the counterexamples due to Pli\'{s} \cite{Pli63} show that we must assume that $A$ is Lipschitz in order for a variable-coefficient operator to have the SUCP.
All of these results were proved using Carleman estimates.

In this article, we focus on generalized Schr\"odinger operators of the form 
\begin{equation}
\label{ellipticOp}
L := - \di\pr{A \gr} + W \cdot \gr + V
\end{equation}
in $\Om \su \R^n$, where $n \ge 2$.
For the coefficients, we assume that $A$ is bounded, elliptic, symmetric, and Lipschitz continuous, and we assume that $W$ belongs to $L^\iny$ and $V$ belongs to $L^p$ for some $p \in [n, \iny]$.
By the results of {\cite{KT01}}, the SUCP holds in this setting.

From a local perspective, we establish bounds for the {\em order of vanishing} of solutions to $L u = 0$.
That is, if $u$ is a solution to $L u = 0$, we seek a lower bound of the form
$$\norm{u}_{L^\iny\pr{B_r}} \gtrsim r^{\be} \;\; \text{ as } \;\; r \to 0,$$
where $\be$ is some function that encapsulates information about the operator, $L$.
{To make the above estimate meaningful, we additionally assume that $u$ is normalized.}
The order of vanishing results may be interpreted as a quantification of the strong unique continuation property.
Theorems \ref{OofV1} and \ref{OofV} below describe such quantifications.

We also study the global unique continuation properties of solutions to $Lu = 0$ in $\R^n$ by quantifying the {\em rate of decay at infinity}.
Such estimates can be interpreted as quantitative Landis-type theorems.
In the late 1960s, E.~M.~Landis \cite{KL88} made the following qualitative unique continuation conjecture: If $u$ is a bounded solution to $-\LP u + V u = 0$ in $\R^n$, where $V$ is a bounded function and $u$ satisfies $\abs{u(x)} \lesssim \exp\pr{- c \abs{x}^{1+}}$, then $u \equiv 0$.

Meshkov \cite{M92} disproved Landis' conjecture by constructing non-trivial $\C$-valued functions $u$ and $V$ that solve $-\LP u + V u = 0$ in $\R^2$, where $V$ is bounded and $\abs{u(x)} \lesssim \exp\pr{- c \abs{x}^{4/3}}$. 
Meshkov also used Carleman estimates to prove a \textit{qualitative unique continuation} result: 
If $-\LP u + V u = 0$ in $\R^n$, where $V$ is bounded and $u$ satisfies a decay estimate of the form $\abs{u(x)} \lesssim \exp\pr{- c \abs{x}^{4/3+}}$, then necessarily $u \equiv 0$.
In their work on Anderson localization \cite{BK05}, Bourgain and Kenig established a quantitative version of Meshkov's result. 
As a first step in their proof, they used three-ball inequalities derived from Carleman estimates to establish {order of vanishing} estimates for local solutions to Schr\"odinger equations.
{Their results show that the order of vanishing for solutions to $-\LP u + V u = 0$ is less than $C\pr{1 + \norm{V}^{2/3}_{L^{\iny}}}$, \cite[pg. 4]{Ken06}.}
Then, through a scaling argument, they proved a {quantitative unique continuation} result.
More specifically, they showed that if $u$ and $V$ are bounded, and $u$ is normalized so that $\abs{u(0)} \ge 1$, then for sufficiently large values of $R$, $\disp  \inf\set{\norm{u}_{L^\iny\pr{B(x_0, 1)}}: |x_0| = R} \ge \exp{(-CR^{4/3}\log R)}$.

In Theorem \ref{UCatIny} below, we reprove the result from \cite{BK05}.
While the results of \cite{M92} and \cite{BK05} rely on Carleman estimates, we use frequency functions.
Theorem \ref{UCatInyAWV} generalizes Theorem \ref{UCatIny} to more general operators and gives a special case of (a weaker version of) the main theorem in \cite{LW14}.

While Carleman estimates have been used to establish all of the unique continuation results discussed so far, another approach is based on {\em frequency functions}.
Almgren \cite{Alm79,Alm00} first showed that the quantity 
$$ N(r) := r \frac{\int_{B_r} \abs{\gr u}^2}{\int_{\del B_r} \abs{u}^2},$$ 
called a frequency function, is non-decreasing in $r$ if $u$ is a harmonic function.
Moreover, if $u$ is a harmonic polynomial, then $N(r)$ is constant and equal to the degree.
Garofalo and Lin \cite{GL86, GL87} used similar frequency functions to establish unique continuation results for solutions to variable-coefficient elliptic equations.
Kukavica \cite{Kuk98} then built on the ideas of Garofalo and Lin and proved order of vanishing estimates for solutions to $-\LP_M u + V u = 0$.
{When the coefficients are bounded, elliptic, symmetric, and Lipschitz continuous, \cite[Theorem 5.1]{Kuk98} shows that the order of vanishing is less than $C\brac{1 + \pr{\sup V_-}^{1/2} + \pr{\osc V}^2}$ if $V \in L^\iny$ or $C\brac{1 + \pr{\sup V_-}^{1/2} + \pr{\osc V} + \norm{\gr V}_\iny}$ if $V \in W^{1,\iny}$.}
In \cite{Kuk00}, Kukavica introduced modified frequency functions of the form
$$N(r) := \frac{\int_{B_r} \abs{\gr u(x)}^2 \pr{r^2 - \abs{x}^2}^{\al}dx}{\int_{B_r} \abs{u(x)}^2 \pr{r^2 - \abs{x}^2}^{\al -1}dx},$$
where $\al \ge 1$.
Using these modified frequency functions, Zhu \cite{Zhu16} {proved that the order of vanishing for solutions to $-\LP u + V u = 0$ is less than $C\pr{1+ \norm{V}^{1/2}_{W^{1,\iny}}}$ if $V \in W^{1,\iny}$} and extended the ideas to higher-order elliptic operators.
In \cite{BG16}, Banerjee and Garofalo used the modified frequency functions to establish quantitative uniqueness for variable-coefficient elliptic equations at the boundary of Dini domains.
Chen and Liu \cite{CL25} have also used modified frequency functions to study quantitative uniqueness for equations with inverse square potentials.
The estimates of \cite{Zhu16} and \cite{BG16} apply to $V \in W^{1,\iny}$ and are sharp.

When $V \in L^\iny$, the order of vanishing estimate in \cite{Kuk98}, {$C\pr{1 + \pr{\sup V_-}^{1/2} + \pr{\osc V}^2}$, is potentially larger than the one implied by the results in \cite{BK05}, $C \pr{1 + \norm{V}^{2/3}_{L^{\iny}}}$}.
That is, the results established using frequency functions aren't as sharp as those obtained via Carleman estimates.
In this article, we show how the modified frequency functions may be used to obtain the same quantitative unique continuation results that have been previously proved via Carleman estimates.
For example, we use frequency functions to prove the order of vanishing estimate for solutions to $-\LP u + V u = 0$ that appears in \cite{BK05}.
If $V \in W^{1,\iny}$, then the divergence theorem may applied to the terms involving $\LP u = V u$ and a careful analysis then leads to an optimal order of vanishing estimate{, as is shown in \cite{Zhu16}}.
Here, with $V \in L^\iny$, we can't apply the divergence theorem to these terms, so we instead complete the square to derive an almost monotonicity result.

To the best of our knowledge, frequency functions have not been previously used to establish the elliptic theorems of this paper.
However, we point out that frequency functions were used by Camliyurt and Kukavica in \cite{CK18} to study quantitative uniqueness of solutions to parabolic equations.
Moreover, their estimates match those of the elliptic setting.

By the constructions that appear in \cite{M92}, Theorem \ref{UCatIny} is sharp when we restrict to $\C$-valued equations and solutions.
The additional constructions that appear in \cite{CS97} and \cite{Dav14} imply that Theorem \ref{UCatInyAWV} is also sharp in the $\C$-valued setting when $p = \iny$.
However, the resolution to Landis' conjecture in the real-valued planar setting {\cite{LMNN25}} shows that Theorem \ref{UCatIny} is not sharp in that setting.

{Although we focus here on Euclidean domains, there are some interesting recent developments in periodic domains. 
The examples in \cite{FK24} establish that a version of Theorem \ref{UCatIny} is sharp for $\R$-valued equations on cylindrical domains of the form $\mathbb{T}^d \times \R$, where $d \ge 3$.
That is, the power of $4/3$ is sharp for such cylindrical domains.
}

Given the relationship between local and global unique continuation estimates, we may similarly conclude that Theorem \ref{OofV} is optimal (at least with respect to the dependencies on $M$ and $K$) in the complex setting when $p = \iny$.
When $p \in [n, \iny)$, the results of \cite{DZ18}, \cite{DZ19}, \cite{Dav20b} show Theorem \ref{OofV} is not optimal.
However, the dependency on $M$ in Theorem \ref{OofV} precisely matches the order of vanishing that was established in \cite{KT16} using $L^2-L^2$ Carleman estimates.

We use the notation $B_r(x)$ to denote a ball of radius $r$ centered at the point $x$, abbreviated by $B_r$ when the center is clear.
Generic constants are denoted by $c, C$ and may change from line to line without comment.
Specific constants may be indicated by subscripts.

The article is organized as follows.
In the next section, Section \ref{S:Statements}, we precisely state the four theorems that are proved in this paper.
In Section \ref{S:ClassicS}, we focus on classical Schr\"odinger operators of the form $L := - \LP + V$.
We introduce the associated frequency function, estimate its derivative, establish a monotonicity result for the frequency function, and then use the monotonicity result to prove a three-balls theorem.
Section \ref{S:GeneralS} mimics Section \ref{S:ClassicS}, except we now work with variable coefficients, a bounded first-order term, and a singular zeroth-order term, so the arguments are more complicated.
Finally, in Section \ref{S:Proofs}, we present the proofs of the theorems that are stated in Section \ref{S:Statements}.

\subsection*{Acknowledgements}
I'd like to thank Agnid Banerjee for posing this question to me during a recent visit to Arizona State University and for interesting mathematical discussions.
{Thank you to the referee for their careful reading and helpful remarks.}

\section{Theorem Statements}
\label{S:Statements}

In this section, we state the main theorems of the article.
There are four theorems in total, two order of vanishing estimates and two Landis-type results.
The first pair of theorems applies to the classical Schr\"odinger operators of the form $-\LP + V$, where $V$ is assumed to be bounded.
The second pair of theorems applies to very general Schr\"odinger operators.
Although the second pair of theorems still applies to the classical operators (and hence subsumes the first pair of theorems), we have chosen to include the first two theorems for readability and comparison with the literature.
All of the proofs are presented in Section \ref{S:Proofs}.

Our first result is an order of vanishing estimate for solutions to Schr\"odinger equations with bounded potentials and may be found in \cite{BK05}.

\begin{theorem}[Order of vanishing estimate I]
\label{OofV1}
Let $V \in L^\iny(B_{10})$ satisfy $\norm{V}_{L^\iny\pr{B_{10}}} \le M$.
Assume that $u$ is a solution to 
$$- \LP u + V u= 0 \; \text{ in } \, B_{10}$$
that is bounded in the sense that
\begin{equation}
\label{uBound}
\norm{u}_{L^\iny(B_{10})} \le C_0
\end{equation}
and normalized so that
\begin{equation}
\label{uNorm}
\norm{u}_{L^2(B_{1})} \ge 1.
\end{equation}
There exists a constant $C(n, C_0) {\gtrsim 1 + \ln C_0 + n + \log \abs{B_1}}$ so that for any $r \ll 1$, it holds that
\begin{equation}
\label{OofVResult1}
\norm{u}_{L^2(B_r)} \ge r^{C \pr{M^{\frac {2}{3}} + 1}}.
\end{equation}
\end{theorem}

The next result is a Landis-type theorem for Schr\"odinger equations of the form $-\LP u + V u = 0$.
The original version of this theorem may be found in \cite[Lemma 3.10]{BK05}.

\begin{theorem}[Unique continuation at infinity I]
\label{UCatIny}
Let $V \in L^\iny(\R^n)$ satisfy $\norm{V}_{L^\iny\pr{\R^n}} \le 1$.
Assume that $u$ is a solution to 
$$- \LP u + V u= 0 \; \text{ in } \, \R^n$$
that is bounded in the sense that
\begin{equation}
\label{uBound1}
\abs{u(x)} \le \exp\pr{C_0 \abs{x}^{\frac{4}{3}}}
\end{equation}
and normalized so that
\begin{equation}
\label{uNorm1}
\norm{u}_{L^2(B_{1})} \ge 1.
\end{equation}
There exists a constant $C(n, C_0)  {\gtrsim 1 + C_0 + n + \log \abs{B_1}}$ so that for any $R \gg 1$, if $\abs{x_0} = R$, then
\begin{equation}
\label{UCatInyResult}
\norm{u}_{L^2(B_1(x_0))} \ge \exp\pr{- C R^{\frac{4}{3}} \log R}.
\end{equation}
\end{theorem}

The next result is a local order of vanishing estimate for solutions to generalized Schr\"odinger equations {for operators of the form \eqref{ellipticOp}}.
Versions of the following result appear throughout the literature; see, for example, \cite[Theorem 1]{DZ18}, \cite[Theorem 1]{DZ19} (when $A = I$, $\la = 1$, $\eta = 0$, $p = \iny$) and \cite{KT16} (when $A = I$, $\la = 1$, $\eta = 0$, $W \equiv 0$){, both obtained using Carleman estimate techniques.
For the variable-coefficient setting, \cite[Theorem 5.1]{Kuk98} uses frequency functions to prove an order of vanishing estimate when $V \in L^\iny$ or $V \in W^{1,\iny}$ (and $W \equiv 0$), but with different dependencies on the norm of $V$ and implicit dependencies on $A$.
In the constant-coefficient setting, when $p \in [n, \iny)$, a sharper order of vanishing estimate is available in  \cite[Theorem 1]{DZ18}, \cite[Theorem 1]{DZ19}. 
The author is not aware of an order of vanishing result that handles variable-coefficients, non-zero first-order terms, and singular zeroth-order terms simultaneously.
In fact, the following result} with the explicit dependence on the Lipschitz constant $\eta$ could be new. 

\begin{theorem}[Order of vanishing estimate II]
\label{OofV}
Let $A : \R^n \to \mathbb{M}_n$ be a symmetric, bounded, uniformly elliptic matrix-valued function with Lipschitz continuous coefficients.
That is, with $A = (a_{ij})_{i,j = 1}^n$, there exists $\la, \eta > 0$ so that 
\begin{align}
&a_{ij}(x) = a_{ji}(x) \qquad \qquad\quad\,\,\, \text{ for every } x \in \R^n, i, j = 1, \ldots, n 
\label{symm} \\
&a_{ij}(x) \xi_i \zeta_j \le \la^{-1} \abs{\xi} \abs{\zeta} \,\qquad  \text{ for every } x \in \R^n, \xi, \zeta \in \R^n 
\label{Abound} \\
&a_{ij}(x) \xi_i \xi_j \ge \la \abs{\xi}^2 \qquad\qquad \text{ for every } x \in \R^n, \xi \in \R^n 
\label{ellip} \\
&\abs{\gr a_{ij}(x)} \le \eta \,\,\,\quad\quad\quad\quad\quad\,  \text{ for a.e. } x \in \R^n .
\label{LipCond}
\end{align}
Let $u$ be a solution to 
$$- \di\pr{A \gr u} + W \cdot \gr u + V u= 0 \; \text{ in } \, B_{10},$$
where $\norm{W}_{L^\iny(B_{10})} \le K$ and $\norm{V}_{L^p\pr{B_{10}}} \le M$ for some $p \in \brac{n, \iny}$.
Assume that $u$ is bounded in the sense of \eqref{uBound} and normalized as in \eqref{uNorm}.
If $A(0) = I$, then there exist constants $c(n, \la), C(n, \la, p, C_0), r_0(n, \la, \eta, p) > 0$ so that for any $r \le r_0$, it holds that
\begin{equation}
\label{OofVResult}
\norm{u}_{L^2(B_r)} \ge r^{C \brac{K^2 + M^{\frac{2p}{3p - 2n}} + \eta + 1} e^{c \eta}}
\end{equation}
\end{theorem}

\begin{remark}
The assumption that $A(0) = I$ is not completely necessary since we could change variables and allow the constants to depend further on $\la$.
However, for convenience, we add this assumption to the theorem.
\end{remark}

Finally, we also establish unique continuation at infinity results for solutions to generalized Schr\"odinger equations.
{When $p = \iny$,} the next result is a special case of (a weaker version of) \cite[Theorem 1.1]{LW14}.
{Our theorem could be viewed as complementary to \cite[Theorem 1.1]{LW14} since here we assume that $V$ is singular, whereas in \cite{LW14}, $V$ is assumed to have pointwise decay at infinity.
As in the local setting, for the constant-coefficient setting, versions of the following result appear in \cite[Theorem 2]{DZ18} and \cite[Theorem 2]{DZ19}  for $p = \iny$.
When $A = I$ and $W = 0$, sharper bounds are available in  \cite[Theorem 4]{DZ18}, \cite[Theorem 4]{DZ19}.
The novelty of the following result is that it applies to operators with variable-coefficients, non-zero first-order terms, and singular zeroth-order terms.}
The decay assumption on $\gr A$ described by \eqref{LipCond2} has appeared previously in \cite{Tu10} and \cite{LW14}.

\begin{theorem}[Unique continuation at infinity II]
\label{UCatInyAWV}
Let $A : \R^n \to \mathbb{M}_n$ be a symmetric, bounded, uniformly elliptic matrix-valued function that satisfies \eqref{symm}, \eqref{Abound}, and \eqref{ellip}.
Assume further that $A = (a_{ij})_{i,j=1}^n$ has Lipschitz continuous coefficients and for some $\eps > 0, \eta \ge 0$, 
\begin{align}
&\abs{\gr a_{ij}(x)}\le \eta \innp{x}^{- (1 + \eps)} \, \text{ for a.e. } x \in \R^n .
\label{LipCond2}
\end{align}
Let $u$ be a solution to 
$$- \di\pr{A \gr u} + W \cdot \gr u + V u= 0 \; \text{ in } \, \R^n,$$
where $\norm{W}_{L^\iny(B_{10})} \le K$ and $\norm{V}_{L^p\pr{B_{10}}} \le M$ for some $p \in \brac{n, \iny}$.
Assume that $u$ is bounded in the sense that
\begin{equation}
\label{uBound2}
\abs{u(x)} \le \exp\pr{C_0 \abs{x}^{2}}
\end{equation}
and normalized so that
\begin{equation}
\label{uNorm2}
\norm{u}_{L^2(B_{1})} \ge 1.
\end{equation}
For any $\de \in (0, \eps)$, there exist constants $C(n, \la, \eta, \eps, p, C_0, \de), R_0(n, \la, \eta, \eps, p, C_0, \de) > 0$ so that if $\abs{x_0} \ge R_0$, then
\begin{equation}
\label{UCatInyResultII}
\norm{u}_{L^2(B_1(x_0))} \ge \exp\pr{- C \abs{x_0}^{2\pr{1+\de}}}.
\end{equation}
\end{theorem}

\begin{remark}
If $W \equiv 0$, then the power of $2$ in the above result may be replaced with $\frac {4p-2n}{3p-2n} \le 2$.
\end{remark}

\begin{remark}
As shown in \cite{LW14}, the result above may be sharpened in the sense that the term $\abs{x_0}^{2\de}$ may be replaced with a function of $\log \abs{x_0}$.
However, to more clearly illustrate our proof techniques, we choose to prove the above (non-sharp) version of this result here.
\end{remark}

All four of these theorems follow from three-balls inequalities that are proved using frequency functions.
Theorems \ref{OofV1} and \ref{UCatIny} are proved with the three-ball inequalities described by Proposition \ref{threeBallsV}, which are established using the frequency functions introduced in Lemma \ref{monoLemma}.
Similarly, Theorems \ref{OofV} and \ref{UCatInyAWV} are proved using the three-ball inequalities described by Proposition \ref{threeBallsA} which rely on the frequency functions introduced in Lemma \ref{monoLemmaA}.
The presence of variable coefficients makes the analysis in the proof of Lemma \ref{monoLemmaA} more delicate, while the presence of (singular) lower-order terms requires some optimization of inequalities, as shown in Lemma \ref{someBoundsLemma}.
All of the proofs of the theorems stated in this section are presented in Section \ref{S:Proofs}.

The proof of Theorem \ref{UCatInyAWV} is similar to that of Theorem \ref{UCatIny}; however, it requires an iterative argument to handle the variable coefficients.
{Roughly speaking, the variability of the coefficient matrix makes the estimates for the rate of decay worse than they would be in the constant-coefficient case.
However, since $\abs{\gr A}$ is assumed to decay at infinity, then each application of our argument leads to an improved bound.
Therefore, by iterating the general proof scheme, we can get arbitrarily close to the optimal bound.}
These details appear at the end of Section \ref{S:Proofs}.

Although {some of} these theorems are not new, to the best of our knowledge, they have not been previously proved using frequency function techniques.
Therefore, the approach here may be new.
Moreover, as can be seen in Corollary \ref{monoCorA}, for example, the frequency function approach shows explicitly how and where each term from the generalized Schr\"odinger operator contributes to our monotone functions. 
This insight alone may be considered interesting and worthwhile.

\section{Classical Schr\"odinger operators}
\label{S:ClassicS}

To illustrate the main ideas, we first consider Schr\"odinger equations of the form $- \LP u + V u = 0$, where $V$ is bounded.
First, we introduce the weight function that appears {in} our frequency functions.
Then we define the frequency function and find a formula for its derivative.
Next, we show that an associated function is monotone non-decreasing whenever $u$ is a solution.
The monotonicity result is then used to establish a three-balls inequality.

We begin by introducing the weight function.
For each $r > 0$, let
\begin{equation}
\label{weightF}
\om_r(x) = r^2 - \abs{x}^2
\end{equation}
Notice that $\om_r(x)$ vanishes along $\del B_r$.
Therefore, if we use this weight function to define our frequency function, we can eliminate all of the boundary integrals that appear when we integrate by parts.
The following lemma will be used it in our computations below.

\begin{lemma}[Weight function derivatives]
\label{derivLem}
Let $\om_r$ be as in \eqref{weightF}.
If $\disp F(r) := \int_{B_r} f(x) \, \om_r(x)^\al dx$ for some $\al \ge 1$, then
\begin{align*}
F'(r) 
&= \frac{2\al + n}{r} F(r) + \frac 1 r \int_{B_r} \gr f(x) \cdot x \, \om_r(x)^{\al} dx \\
&= \frac{2\al + n}{r} F(r) + \frac 1 {2 (\al + 1)r} \int_{B_r} \LP f(x)  \om_r(x)^{\al+1} dx.
\end{align*}
\end{lemma}

\begin{proof}
Differentiating $F(r)$ with respect to $r$ gives
\begin{align*}
F'(r) 
&= 2 \al r \int_{B_r} f(x) \om_r(x)^{\al-1} dx 
= \frac{2 \al} r \int_{B_r} f(x) \om_r(x)^{\al-1} \pr{r^2 - \abs{x}^2 + \abs{x}^2}dx \\
&= \frac{2 \al} r F(r) + \frac {2\al} r \int_{B_r} f(x) \abs{x}^2  \om_r(x)^{\al-1} dx.
\end{align*}
For the second term, we notice that 
\begin{align*}
2\al \int_{B_r} f(x) \abs{x}^2  \om_r(x)^{\al-1} dx
&=  \int_{B_r} f(x) x \cdot 2 \al x \, \om_r(x)^{\al-1} dx
= - \int_{B_r} f(x) x \cdot \gr \brac{ \om_r(x)^{\al}} dx \\
&= \int_{B_r} \di \brac{f(x) x} \om_r(x)^{\al} dx
= n F(r) + \int_{B_r} \gr f(x) \cdot x \, \om_r(x)^{\al} dx,
\end{align*}
where we have integrated by parts.
Integrating by parts again shows that
\begin{align*}
\int_{B_r} \gr f(x) \cdot x \, \om_r(x)^{\al} dx
&= \frac 1 {2 (\al + 1)} \int_{B_r} \gr f(x) \cdot 2 (\al + 1) x \, \om_r(x)^{\al} dx \\
&= \frac 1 {2 (\al + 1)} \int_{B_r} \LP f(x)  \om_r(x)^{\al+1} dx.
\end{align*}
Combining all of these computations leads to the conclusion.
\end{proof}

Now we use the weight function from above to define our frequency function, then we estimate its derivative.
The new idea in this proof is to complete the square, see \eqref{Dderiv}.

\begin{lemma}[Frequency function]
\label{monoLemma}
For some $R > 0$, let $u \in H^2(B_R)$.
Fix $\al \ge 2$.
For $r \in (0, R)$, define
\begin{equation}
\label{HDLDefn}
\begin{aligned}
H(r) &:= \int_{B_r} \abs{u(x)}^2  \om_r(x)^{\al-1} dx \\
D(r) &:= \int_{B_r} \abs{\gr u(x)}^2  \om_r(x)^{\al} dx \\
L(r) &:= \int_{B_r} u(x) \LP u(x) \, \om_r(x)^{\al} dx,
\end{aligned}
\end{equation}
where $\om_r$ is from \eqref{weightF}.
Then
\begin{equation}
\label{Hderiv}
H'(r) =  \frac{2 (\al-1) + n}{r} H(r) + \frac {D(r) + L(r)}{\al r}
\end{equation}
and with
\begin{equation}
\label{NDefn}
N(r) := \frac{D(r)}{H(r)}
\end{equation}
it holds that
\begin{equation}
\label{N'Bound}
N'(r) 
\ge - \frac{1}{4 \al r H(r)} \int_{B_r} \abs{\LP u(x)}^2 \om_r(x)^{\al+1}dx.
\end{equation}
\end{lemma}

\begin{proof}
With $H(r)$ as given, an application of Lemma \ref{derivLem} with $f(x) = \abs{u(x)}^2$ shows that
\begin{align*}
H'(r) 
&=  \frac{2 (\al-1) + n}{r} H(r) + \frac 1 {\al r} I(r),
\end{align*}
where
\begin{equation}
\label{IDefn}
\begin{aligned}
I(r) &:= 2 \al  \int_{B_r} u \pr{x \cdot \gr u} \om_r^{\al-1}
= D(r) + L(r).
\end{aligned}
\end{equation}
In particular, \eqref{Hderiv} has been shown.
Another application of Lemma \ref{derivLem}, now with $f(x) = \abs{\gr u(x)}^2$, shows that
\begin{align*}
D'(r) 
&= \frac{2 \al + n}{r} \int_{B_r} \abs{\gr u}^2 \om_r^{\al}
+ \frac{2}{r} \int_{B_r} \sum_{i, j} \del_{ij} u \del_i u x_j \om_r^{\al}.
\end{align*}
For the second term, an integration by parts gives
\begin{align*}
 \int_{B_r} \sum_{i, j} \del_{ij} u \del_i u x_j \om_r^{\al}
 &= -  \int_{B_r} \sum_{i, j} \del_{j} u \del_i \brac{\del_i u x_j \om_r^{\al}} \\
& = -  \int_{B_r} \pr{x \cdot \gr u} \LP u \, \om_r^{\al}
-  \int_{B_r} \abs{\gr u}^2 \om_r^{\al}
+ 2 \al  \int_{B_r} \pr{x \cdot \gr u}^2 \om_r^{\al-1}.
\end{align*}
Therefore,
\begin{equation}
\label{Dderiv}
\begin{aligned}
D'(r) 
&= \frac{2 (\al-1) + n}{r} D(r)
+ \frac{4 \al}{r}  \int_{B_r} \pr{x \cdot \gr u}^2 \om_r^{\al-1}
- \frac{2}{r} \int_{B_r} \pr{x \cdot \gr u} \LP u \, \om_r^{\al} \\
&= \frac{2 (\al-1) + n}{r} D(r)
+ \frac{4\al }{r}  \int_{B_r} \pr{ x \cdot \gr u - \frac{ \LP u \, \om_r}{4\al}}^2 \om_r^{\al-1}
- \frac{1}{4\al r} \int_{B_r} \abs{\LP u}^2 \om_r^{\al+1},
\end{aligned}
\end{equation}
where we have completed the square to reach the second line.
Set 
\begin{equation}
\label{JDefn}
J(r) := \frac {I(r) + D(r)} 2 = 2 \al \int_{B_r} u \pr{ x \cdot \gr u - \frac{ \LP u \, \om_r}{4 \al} } \om_r^{\al-1}
\end{equation}
so that 
\begin{equation}
\label{Jrelat}
\begin{aligned}
D(r) &= J(r) - \frac 1 2 L(r) \\
I(r) &= J(r) + \frac 1 2 L(r).
\end{aligned}
\end{equation}
Using \eqref{Dderiv} and \eqref{Hderiv} with \eqref{IDefn}, then \eqref{Jrelat}, then \eqref{HDLDefn} and \eqref{JDefn}, we see that
\begin{align*}
& H^2(r) N'(r) \\
&= D'(r) H(r) - H'(r) D(r) \\
&= \brac{\frac{2 (\al-1) + n}{r} D(r)
+ \frac{4 \al }{r}  \int_{B_r} \pr{ x \cdot \gr u - \frac{\LP u \, \om_r}{4 \al } }^2 \om_r^{\al-1} - \frac{1}{4 \al r} \int_{B_r} \abs{\LP u}^2 \om_r^{\al+1}} H(r)  \\
&- \brac{\frac{2 (\al-1) + n}{r} H(r) + \frac {I(r)}{ \al r}} D(r) \\
&= \brac{\frac{4 \al }{r}  \int_{B_r} \pr{ x \cdot \gr u - \frac{\LP u \, \om_r}{4 \al } }^2 \om_r^{\al-1} - \frac{1}{4 \al r} \int_{B_r} \abs{\LP u}^2 \om_r^{\al+1}} H(r) \\
&-  \frac 1 {\al r} \brac{J(r) + \frac 1 2 L(r)}\brac{J(r) - \frac 1 2 L(r)}  \\
&= \frac{4 \al }{r} H(r) \int_{B_r} \pr{ x \cdot \gr u - \frac{\LP u \, \om_r}{4 \al }}^2 \om_r^{\al-1}
- \frac{\brac{J(r)}^2}{ \al r}  
+ \frac {\brac{L(r)}^2}{4 \al r} 
- \frac{1}{4 \al r} H(r) \int_{B_r} \abs{\LP u}^2 \om_r^{\al+1}  \\
&= \frac{4 \al }{r} \set{{\int_{B_r} \abs{u}^2 \om_r^{\al-1}} \cdot { \int_{B_r} \pr{ x \cdot \gr u - \frac{\LP u \, \om_r}{4 \al } }^2 \om_r^{\al-1}} 
- \brac{ \int_{B_r} u \pr{ x \cdot \gr u - \frac{\LP u \, \om_r}{4 \al } } \om_r^{\al-1}}^2} \\
&+ \frac 1 {4 \al r} \set{ \brac{\int_{B_r} u \LP u \, \om_r^{\al}}^2 - H(r)  \int_{B_r} \abs{\LP u}^2 \om_r^{\al+1} }.
\end{align*}
An application of Cauchy-Schwarz shows that the first line is non-negative and we get \eqref{N'Bound}.
\end{proof}

\begin{cor}[Monotonicity result]
\label{monoCorVIny}
With $N(r)$ as in \eqref{NDefn}, define
\begin{equation*}
\widetilde{N}(r) := N(r) + \frac{M^2 r^4}{16\al}.
\end{equation*}
If $\norm{V}_{L^\iny\pr{B_R}} \le M$ and $u$ is a solution to $- \LP u + V u= 0$ in $B_R$, then $\widetilde{N}(r)$ is a non-decreasing function of $r$.
\end{cor}

\begin{proof}
Since $-\LP u + V u= 0$ in $B_R$, then
\begin{align*}
\int_{B_r} \abs{\LP u}^2 \om_r^{\al+1} 
&=  \int_{B_r} \abs{V u}^2 \om_r^{\al+1} 
\le M^2 r^4 H(r).
\end{align*}
Substituting this bound into \eqref{N'Bound} from Lemma \ref{monoLemma} shows that $\disp N'(r) \ge - \frac{ M^2 r^4 }{4 \al r}$ and the result follows.
\end{proof}

Using the monotonicity result above and choosing $\al \gg 1$ leads to the following three-balls inequality.

\begin{prop}[Three-ball inequalities]
\label{threeBallsV}
Let $u$ be a solution to $- \LP u + V u = 0$ in $B_R$, where $\norm{V}_{L^\iny(B_R)} \le M$.
For any $0 < r_1 < r_2 < 2r_2 < r_3 < R$, it holds that
\begin{align*}
\norm{u}_{L^2(B_{r_2})} 
&\le \exp\set{C {\brac{1 + \pr{M R^2}^{\frac 2 3}}}} \norm{u}_{L^2(B_{r_1})} ^{\kappa} \norm{u}_{L^2(B_{r_3})} ^{1 - \kappa} ,
\end{align*}
where $C > 0$ is a universal constant and $\disp \kappa := \frac{\log (r_3) - \log\pr{2 r_2}}{\log (r_3) - \log\pr{r_1}}$.
\end{prop}

\begin{proof}
From \eqref{Hderiv}, we have
\begin{align*}
\frac{H'(r) }{H(r)}
&=  \frac{2 (\al-1) + n}{r} + \frac 1 {\al r} \frac{D(r) + L(r)}{H(r)}
=  \frac{2 \al + n - 2}{r} + \frac {N(r)} {\al r} + \frac 1 {\al r} \frac{L(r)}{H(r)},
\end{align*}
where we recall \eqref{NDefn}.
From \eqref{HDLDefn} and the assumption that $\abs{\LP u} = \abs{V u} \le M \abs{u}$, we get $L(r) \le M r^2 H(r)$.
Therefore, 
\begin{align*}
\frac{H'(r) }{H(r)}
&\ge  \frac{2 \al + n-2}{r} + \frac 1 {\al r}\brac{N(r) + \frac {M^2 r^4} {16 \al }} - \frac {M^2 r^3}{16 \al^2} - \frac {Mr} {\al }
\end{align*}
and
\begin{align*}
\frac{H'(r) }{H(r)}
&\le  \frac{2 \al + n-2}{r} + \frac 1 {\al r}\brac{N(r) + \frac {M^2 r^4} {16 \al }} + \frac {Mr} {\al }.
\end{align*}
An application of Corollary \ref{monoCorVIny} shows that the bracketed terms are equal to $\widetilde{N}(r)$ and are non-decreasing.
Integrating the lower bound from $2r_2 < r_3 < R$ shows that
\begin{align*}
\log\brac{\frac{H(r_3)}{H(2r_2)}}
&= \int_{2r_2}^{r_3} \frac{H'(r) }{H(r)} dr
\ge  \int_{2r_2}^{r_3} \set{\frac{2 \al + n-2}{r} + \frac {\widetilde{N}(r)}{\al r} - \frac {M^2 r^3}{16 \al^2}  - \frac {Mr} {\al }} dr \\
&\ge \brac{2 \al + n - 2 + \frac{\widetilde{N}(2r_2)}{\al}} \int_{2r_2}^{r_3} \frac 1 {r} dr
- \frac {M^2 }{64 \al^2} \int_{2r_2}^{r_3}  4 r^3 dr
- \frac {M} {2 \al} \int_{2r_2}^{r_3} 2 rdr \\
&\ge \brac{2 \al + n - 2 + \frac{\widetilde{N}(2r_2)}{\al} }  \log\pr{\frac{r_3}{2 r_2}} 
- \frac {M^2 r_3^4}{64 \al^2} 
- \frac {M r_3^2} {2\al } .
\end{align*}
On the other hand, integrating the upper bound from $0 < r_1 < 2 r_2$ gives
\begin{align*}
\log\brac{\frac{H(2r_2)}{H(r_1)}}
&= \int_{r_1}^{2r_2} \frac{H'(r) }{H(r)} dr 
\le \int_{r_1}^{2r_2} \set{\frac{2 \al + n-2}{r} + \frac{\widetilde{N}(r)}{\al r} + \frac {Mr} {\al }} dr \\
&\le \brac{2 \al + n -2 + \frac{\widetilde{N}(2r_2)}{\al }} \log\pr{\frac{2r_2}{r_1}} 
+ \frac {M \pr{2 r_2}^2} {2\al}.
\end{align*}
Combining these inequalities shows that 
\begin{align*}
\frac{\log\brac{\frac{H(2r_2)}{H(r_1)}} -  \frac {M \pr{2 r_2}^2} {2\al}}{\log\pr{\frac{2r_2}{r_1}} }
&\le 2 \al + n -2 + \frac{\widetilde{N}(2r_2)}{\al } 
\le \frac{\log\brac{\frac{H(r_3)}{H(2r_2)}} +\frac {M^2 r_3^4}{64 \al^2} + \frac {M r_3^2} {2\al } }{\log\pr{\frac{r_3}{2 r_2}} }.
\end{align*}

For $r \in (0, R)$, let
\begin{equation}
\label{hDefn}
h(r) := \int_{B_r} \abs{u(x)}^2 dx.
\end{equation}
Observe that $H(r) \le r^{2(\al-1)} h(r)$ while for any $0 < r < \rho \le 1$, $h(r) \le \frac{H(\rho)}{\pr{\rho^2 - r^2}^{\al-1}}$.
In particular, 
\begin{equation*}
H(2 r_2) \ge \pr{3 r_2^2}^{\al-1} h(r_2) = \pr{\frac 3 4}^{\al-1} \pr{2 r_2}^{2\pr{\al-1}} h(r_2).
\end{equation*}
Therefore,
\begin{equation}
\label{hObs}
\begin{aligned}
\log\brac{\frac{H(r_3)}{H(2r_2)}}
&\le \log\brac{\frac{r_3^{2\pr{\al-1}} h(r_3) \pr{\frac 4 3}^{\al-1} }{\pr{2 r_2}^{2\pr{\al-1}}  h(r_2)}}  \\
&= \log\brac{\frac{h(r_3)}{h(r_2)}} + \pr{\al-1}\brac{2 \log\pr{\frac{r_3}{2 r_2}} + \log\pr{\frac 4 3}} \\
\log\brac{\frac{H(2r_2)}{H(r_1)}}
&\ge \log\brac{\frac{\pr{2 r_2}^{2\pr{\al-1}} h(r_2)}{r_1^{2\pr{\al-1}} h(r_1) \pr{\frac 4 3}^{\al-1} }} \\
&= \log\brac{\frac{h(r_2)}{h(r_1)}} + \pr{\al-1} \brac{2 \log\pr{\frac{2 r_2}{r_1}} - \log\pr{\frac 4 3}}.
\end{aligned}
\end{equation}
Substituting these bounds into the expression above gives
\begin{align*}
\frac{\log\brac{\frac{h(r_2)}{h(r_1)}} - \pr{\al-1} \log\pr{\frac 4 3} -  \frac {M \pr{2 r_2}^2} {2\al}}{\log\pr{\frac{2r_2}{r_1}} }
&\le \frac{\log\brac{\frac{h(r_3)}{h(r_2)}} + \pr{\al-1} \log\pr{\frac 4 3} +\frac {M^2 r_3^4}{64 \al^2} + \frac {M r_3^2} {2\al } }{\log\pr{\frac{r_3}{2 r_2}} }.
\end{align*}
Set $\be = \log\pr{\frac{2r_2}{r_1}}$ and $\ga = \log\pr{\frac{r_3}{2 r_2}} $, then simplify to get
\begin{align*}
 \log h(r_2) 
&\le \log h(r_1)^{\frac{\ga}{\be + \ga}} + \log h(r_3)^{\frac{\be}{\be + \ga}} 
+ (\al-1) \log\pr{\frac 4 3}  + \frac {M R^2} {2\al} + \frac {M^2 R^4}{64 \al^2}.
\end{align*}
Exponentiating then shows that
\begin{align*}
h(r_2) 
&\le \exp\brac{\al \log\pr{\frac 4 3}  + \frac {M R^2} {2\al} + \pr{\frac{M R^2}{8 \al}}^2 } h(r_1)^{\kappa} h(r_3)^{1 - \kappa} ,
\end{align*}
where $\disp \kappa = \frac{\log (r_3) - \log\pr{2 r_2}}{\log (r_3) - \log\pr{r_1}}$.
The exponential term is optimized when $\al \simeq \pr{M R^2}^{\frac 2 3}$, {so we choose $\al = 2\brac{1 + \pr{MR^2}^{2/3}} \ge 2$,} which leads to the conclusion.
\end{proof}

\section{Generalized Schr\"odinger operators}
\label{S:GeneralS}

In this section, we consider much more general Schr\"odinger operators with variable coefficients, a first-order term, and a zeroth-order term that may be singular.
Consequently, we begin with a technical lemma that helps us estimate the derivative of our variable-coefficient frequency function.
Then we define the frequency function and estimate its derivative by applying the previous lemma and completing the square.
From here, we find the associated function that is monotone non-decreasing whenever $u$ is a solution.
As in the previous section, the monotonicity result leads to a three-ball inequality if we choose $\al$ sufficiently large.
Given the presence of variable coefficients and lower order terms, the arguments here are somewhat more complicated, but we follow the same scheme from the previous section.

Throughout this section, we assume that $A : \R^n \to \mathbb{M}_n$ is a symmetric, bounded, uniformly elliptic matrix with Lipschitz continuous coefficients.
That is, with $A = \pr{a_{ij}}_{i,j = 1}^n$, there exists $\la, \eta > 0$ so that \eqref{symm} -- \eqref{LipCond} hold.
Before proceeding to the frequency functions, we'll collect some estimates for functions associated to the matrix $A$.
The function $\mu$ below was introduced in \cite{GL86} and has been used in many instances since, while the vector field $Z(x)$ appears in \cite{GPSVG18}, for example.

\begin{lemma}[Variable-coefficient estimates]
\label{AestLemma}
Let $A : \R^n \to \mathbb{M}_n$ satisfy \eqref{symm} -- \eqref{LipCond}.
Define the function $\mu : \R^n \setminus \set{0} \to [\la, \la^{-1}]$ by
\begin{equation}
\label{muDefn}
\mu(x) = \frac{A(x) x \cdot x}{\abs{x}^2}
\end{equation}
and the vector field $Z : \R^n \setminus \set{0} \to \R^n$ by
\begin{equation}
\label{ZDefn}
Z(x) = \frac{A(x) x}{\mu(x)}.
\end{equation}
If $A(0) = I$, then there exist constants $c_1(n, \la), c_2(n, \la), c_3(n, \la) > 0$ so that for a.e. $x \in B_r$,
\begin{align}
&\abs{\di \brac{A(x) x}\mu(x)^{-1} - n }
\le c_1 \eta r
\label{diAnmu} \\
&\abs{\de_{ik} - \del_i Z_k(x)} \le c_2 \eta r 
\label{delZkcomp} \\
&\abs{\di Z(x) - n} \le c_3 \eta r.
\label{divZest}
\end{align}
\end{lemma}

\begin{proof}
Since $A(0) = I$, then 
\begin{align*}
\di \brac{A(x) x}
&= \sum_{i, j} \del_i \brac{a_{ij}(x) x_j}
= \sum_{i, j} \del_i a_{ij}(x) x_j + \sum_{i} a_{ii}(x) \\
&= \sum_{i, j} \del_i a_{ij}(x) x_j + \sum_{i} \brac{a_{ii}(x) -a_{ii}(0)} + n
\end{align*}
and
\begin{align*}
\mu(x) 
&= \frac{ \sum_{i, j} a_{ij}(x) x_i x_j }{\abs{x}^2}
= \frac{\sum_{i, j} \brac{a_{ij}(x) - a_{ij}(0) + \de_{ij}} x_i x_j}{\abs{x}^2}
= 1 +  \frac{\sum_{i, j} \brac{a_{ij}(x) - a_{ij}(0)} x_i x_j}{\abs{x}^2}.
\end{align*}
Therefore,
\begin{align*}
\di \brac{A(x) x} - n \mu(x)
&= \sum_{i, j} \del_i a_{ij}(x) x_j + \sum_{i} \brac{a_{ii}(x) -a_{ii}(0)} - n \frac{\sum_{i, j} \brac{a_{ij}(x) - a_{ij}(0)} x_i x_j}{\abs{x}^2}
\end{align*}
and then the bounds on $\mu(x)$, the mean value theorem, and \eqref{LipCond} imply that there exists a constant $c_1 > 0$ so that for a.e. $x \in B_r$, \eqref{diAnmu} holds.
Since $A(0) = I$, then
\begin{align*}
\brac{A(x) x}_k
&= \sum_{j} a_{kj}(x) x_j
= \sum_{j} \brac{a_{kj}(x) - a_{kj}(0)} x_j + x_k \\
\del_i \brac{A(x) x}_k
&= \sum_{j} \del_i \brac{a_{kj}(x) x_j}
%= \sum_{j} \del_i a_{kj}(x) x_j + a_{ki}(x)
= \sum_{j} \del_i a_{kj}(x) x_j + \brac{a_{ki}(x) - a_{ki}(0)} + \de_{ki} \\
\del_i \brac{A(x) x \cdot x}
&= \sum_{\ell,m} \del_i \brac{a_{\ell m}(x) x_\ell x_m}
= \sum_{\ell,m} \del_i a_{\ell m}(x) x_\ell x_m
+ 2 \sum_{m} \brac{a_{im}(x) -a_{im}(0)} x_m + 2 x_i,
\end{align*}
where we have used the symmetry of $A$.
From here, we get that
\begin{align*}
\del_i Z_k
&= \del_i \brac{ \abs{x}^2 \frac{\brac{A(x) x}_k}{ A(x) x \cdot x}} \\
&= 2  \frac{ \sum_{j} \brac{a_{kj}(x) - a_{kj}(0)} x_j x_i + x_k x_i}{ A x \cdot x} 
+ \abs{x}^2  \brac{\frac{\sum_{j} \del_i a_{kj}(x) x_j + \brac{a_{ki}(x) - a_{ki}(0)} + \de_{ki}}{ Ax \cdot x} } \\
&-  \abs{x}^2 \frac{\sum_j \brac{a_{kj}(x) - a_{kj}(0)} x_j }{ \pr{Ax \cdot x}^2} \brac{\sum_{\ell,m} \del_i a_{\ell m}(x) x_\ell x_m + 2 \sum_{m} \brac{a_{im}(x) -a_{im}(0)} x_m + 2 x_i} \\
&-  \abs{x}^2 \frac{ x_k}{ \pr{Ax \cdot x}^2} \brac{\sum_{\ell,m} \del_i a_{\ell m}(x) x_\ell x_m + 2 \sum_{m} \brac{a_{im}(x) -a_{im}(0)} x_m + 2 x_i} .
\end{align*}
After simplifications, we see that
\begin{align*}
\del_i Z_k
&= \de_{ik}
- \frac{ \sum_{j, \ell} \brac{a_{j\ell}(x) - a_{j \ell}(0) } x_j x_\ell}{ Ax \cdot x} \brac{\de_{ik} - \frac{2 x_i }{ Ax \cdot x} \pr{x_k  + \sum_{j} \brac{a_{kj}(x) - a_{kj}(0)} x_j } }
 \\
&+ \abs{x}^2  \brac{\frac{\sum_{j} \del_i a_{kj}(x) x_j + \brac{a_{ki}(x) - a_{ki}(0)} }{ Ax \cdot x} } \\
&
-  \abs{x}^2 \frac{x_k}{ \pr{Ax \cdot x}^2} \brac{\sum_{\ell,m} \del_i a_{\ell m}(x) x_\ell x_m + 2 \sum_{m} \brac{a_{im}(x) -a_{im}(0)} x_m} \\
&
-  \abs{x}^2 \frac{\sum_j \brac{a_{kj}(x) - a_{kj}(0)} x_j}{ \pr{Ax \cdot x}^2} \brac{\sum_{\ell,m} \del_i a_{\ell m}(x) x_\ell x_m
+ 2 \sum_{m} \brac{a_{im}(x) -a_{im}(0)} x_m }.
\end{align*}
Applications of \eqref{ellip}, \eqref{Abound}, \eqref{LipCond}, and the mean value theorem lead to \eqref{delZkcomp}.
Finally, \eqref{delZkcomp} implies \eqref{divZest}.
\end{proof}

Now we introduce the variable-coefficient frequency function and estimate its derivative.
We essentially follow the approach used to prove Lemma \ref{monoLemma}, but we additionally incorporate the previous lemma to account for the variable coefficients.

\begin{lemma}[Variable-coefficient frequency function]
\label{monoLemmaA}
For some $R > 0$, let $u \in H^2(B_R)$.
With $A : B_R \to \mathbb{M}_n$, define $\mu$ as in \eqref{muDefn}.
For $r \in (0, R)$, $\al \ge 2$, and $\om_r$ from \eqref{weightF}, define
\begin{equation}
\label{HDADefn}
\begin{aligned}
H(r) &:= \int_{B_r} \abs{u(x)}^2 \mu(x) \, \om_r(x)^{\al-1}  dx  \\
D(r) &:= \int_{B_r} A(x) \gr u(x) \cdot \gr u(x) \om_r(x)^{\al} dx,
\end{aligned}
\end{equation}
then set 
\begin{equation}
\label{NADefn}
N(r) := \frac{D(r)}{H(r)}.
\end{equation}
If $A$ is a symmetric, bounded, uniformly elliptic matrix-valued function with Lipschitz continuous coefficients that satisfies \eqref{symm} -- \eqref{LipCond}, and $A(0) = I$, then exists a constant $c_0(n, \la) > 0$ so that 
\begin{equation}
\label{NA'Bound}
N'(r) + c_0 \eta N(r)
\ge - \frac{\int_{B_r}\brac{  \di\pr{A \gr u} }^2 \mu^{-1} \om_r^{\al+1}}{4 \al r H(r)} .
\end{equation}
\end{lemma}

\begin{proof}
With $H(r)$ as given, repeating the arguments in the proof of Lemma \ref{derivLem} shows that
\begin{equation}
\label{HADerivComp}
\begin{aligned}
H'(r) - \frac{2 (\al-1)}{r} H(r) 
&=  \frac{2(\al-1)}{r} \int_{B_r} \abs{u}^2 \abs{x}^2 \mu \, \om_r^{\al - 2}
=  \frac{1}{r} \int_{B_r} \abs{u}^2 A x \cdot 2(\al-1) x \, \om_r^{\al - 2} \\
&=  \frac{1}{r} \int_{B_r} \di \pr{\abs{u}^2 A x} \om_r^{\al -1} \\
&=  \frac{n}{r} H(r)
+ \frac{1}{\al r} I(r)
+ \frac{1}{r} \int_{B_r} \abs{u}^2  \brac{\di \pr{A x} - n \mu} \om_r^{\al-1} ,
\end{aligned}
\end{equation}
where another integration by parts shows that
\begin{equation}
\label{IADefn}
\begin{aligned}
I(r) &:= 2 \al \int_{B_r} u A \gr u \cdot x \, \om_r^{\al-1} 
= \int_{B_r} A \gr u \cdot \gr u \, \om_r^{\al}
+ \int_{B_r} u \di \pr{A \gr u} \om_r^{\al} \\
&= D(r) + L(r),
\end{aligned}
\end{equation}
where we introduce
\begin{align}
\label{LAdefn}
&L(r) := \int_{B_r} u \di \pr{A \gr u} \om_r^{\al}.
\end{align}
With
\begin{align}
\label{EHdefn}
&E_H(r) := \int_{B_r} \abs{u}^2  \brac{\di \pr{A x}\mu^{-1} - n} \mu \om_r^{\al-1},
\end{align}
we see that
\begin{equation}
\label{HADeriv}
H'(r) 
= \frac{2 (\al-1) + n}{r} H(r)
+ \frac{D(r) + L(r)}{\al r}
+ \frac{1}{r} E_H(r).
\end{equation}
Since Lemma \ref{AestLemma}  is applicable, then \eqref{diAnmu} shows that
\begin{equation}
\label{EHBound}
E_H(r) \le c_1 \eta r H(r).
\end{equation}

Now we look at the derivative of $D(r)$.
As in \cite{GPSVG18}, we use the vector field $\disp Z(x) = \frac{A x}{\mu(x)}$ from \eqref{ZDefn} and observe that $Z(x) \cdot x = \abs{x}^2$.
Repeating the arguments in the proof of Lemma \ref{derivLem} shows that
\begin{equation}
\label{DA'1}
\begin{aligned}
D'(r) - \frac{2 \al}{r} D(r)
&=  \frac {2\al} r \int_{B_r} A \gr u \cdot \gr u \abs{x}^2 \om_r^{\al-1}
=  \frac {2\al } r \int_{B_r} A \gr u \cdot \gr u \, Z \cdot x \, \om_r^{\al-1} \\
&=  \frac {1} r \int_{B_r}  \di Z A \gr u \cdot \gr u \, \om_r^{\al} 
+  \frac {1} r \int_{B_r}  \gr \pr{A \gr u \cdot \gr u} \cdot Z \, \om_r^{\al}.
\end{aligned}
\end{equation}
For the last term, an integration by parts shows that
\begin{align*}
\int_{B_r} \del_k \pr{a_{ij} \del_i u \del_j u} Z_k \om_r^{\al}
=& \int_{B_r} \del_k a_{ij} \del_i u \del_j u Z_k \om_r^{\al} 
+ \int_{B_r} a_{ij} \pr{\del_{ik} u \del_j u + \del_i u \del_{jk} u} Z_k \om_r^{\al} \\
=& \int_{B_r} \del_k a_{ij} \del_i u \del_j u Z_k \om_r^{\al} 
- 2 \int_{B_r} Z_k \del_{k} u \del_i\pr{a_{ij} \del_j u} \om_r^{\al} \\
-& 2 \int_{B_r} a_{ij} \del_{k} u \del_j u \del_i Z_k \om_r^{\al}
+ 4\al \int_{B_r} a_{ij} \del_j u x_i Z_k \del_{k} u \, \om_r^{\al-1}.
\end{align*}
Therefore,
\begin{align*}
\int_{B_r} & \gr \pr{A \gr u \cdot \gr u} \cdot Z \, \om_r^{\al} 
= - 2 \int_{B_r} A \gr u \cdot \gr u \, \om_r^{\al}
+ 4\al \int_{B_r} \pr{A \gr u \cdot x}^2 \mu^{-1} \om_r^{\al-1} \\
-& 2 \int_{B_r} A \gr u \cdot x \di\pr{A \gr u} \mu^{-1} \om_r^{\al} 
 + 2 \int_{B_r} a_{ij} \del_{k} u \del_j u \pr{\de_{ik} - \del_i Z_k} \om_r^{\al}
+ \int_{B_r} \del_k a_{ij} \del_i u \del_j u Z_k \om_r^{\al}.
\end{align*}
Substituting this expression into \eqref{DA'1} shows that
\begin{equation}
\label{DAderiv}
\begin{aligned}
D'(r)
&= \frac{2(\al-1)  + n}{r} D(r)
+ \frac{4\al }{r} \int_{B_r} \pr{A \gr u \cdot x}^2 \mu^{-1} \om_r^{\al-1}  \\
&- \frac{2}{r} \int_{B_r} A \gr u \cdot x \di\pr{A \gr u} \mu^{-1} \om_r^{\al}
+ \frac {1} r E_D(r) \\
&= \frac{2\al  + n-2}{r} D(r)
+ \frac{4\al }{r} \int_{B_r}\brac{ \pr{A \gr u \cdot x}^2 - \frac{\di\pr{A \gr u} \om_r}{4 \al}  }^2 \mu^{-1} \om_r^{\al-1}  \\
&- \frac{1}{4\al r} \int_{B_r}\brac{  \di\pr{A \gr u} }^2 \mu^{-1} \om_r^{\al+1} 
+ \frac {1} r E_D(r),
\end{aligned}
\end{equation}
where we introduce
\begin{equation}
\label{EDdefn}
\begin{aligned}
E_D(r) 
&:= \int_{B_r}  \brac{\di Z - n} A \gr u \cdot \gr u \, \om_r^{\al}
+ 2 \int_{B_r} a_{ij} \del_{k} u \del_j u \pr{\de_{ik} - \del_i Z_k} \om_r^{\al} \\
&+ \int_{B_r} \del_k a_{ij} \del_i u \del_j u Z_k \om_r^{\al}.
\end{aligned}
\end{equation}
Applications of \eqref{LipCond} as well as \eqref{delZkcomp} and \eqref{divZest} from Lemma \ref{AestLemma} show that for some $c(n, \la) > 0$,
\begin{equation}
\label{EDBound}
\abs{E_D(r)}
\le c \eta r D(r).
\end{equation}

Now set
$$J(r) := \frac {I(r) + D(r)} 2 = 2 \al \int_{B_r} u \brac{ A \gr u \cdot x - \frac{\di\pr{A \gr u} \om_r}{4\al } } \om_r^{\al-1}$$ 
so that with $L(r)$ as in \eqref{LAdefn}, we get
\begin{equation}
\label{JArelat}
\begin{aligned}
D(r) &= J(r) - \frac 1 2 L(r) \\
I(r) &= J(r) + \frac 1 2 L(r).
\end{aligned}
\end{equation}
Using \eqref{HDADefn}, \eqref{HADeriv}, \eqref{DAderiv}, and \eqref{JArelat}, we see that
\begin{align*}
H^2(r) N'(r) 
&= D'(r) H(r) - H'(r) D(r) \\
&= \frac{2\al  + n-2}{r} D(r) H(r)
+ \frac{4\al }{r} H(r) \int_{B_r}\brac{ \pr{A \gr u \cdot x}^2 - \frac{\di\pr{A \gr u} \om_r}{4\al}   }^2 \mu^{-1} \om_r^{\al-1}  \\
&- \frac{1}{4\al r} H(r) \int_{B_r}\brac{  \di\pr{A \gr u} }^2 \mu^{-1} \om_r^{\al+1} 
+ \frac {1} r H(r) E_D(r) \\
&- \frac{2 \al + n-2}{r} H(r) D(r)
- \frac{J(r) + \frac 1 2 L(r)}{\al r} \pr{J(r) - \frac 1 2 L(r)}
- \frac{1}{r} E_H(r)D(r) \\
&= \frac{4 \al }{r} \int_{B_r} \abs{u}^2  \mu \, \om_r^{\al-1} \int_{B_r}\brac{ \pr{A \gr u \cdot x}^2 - \frac{\di\pr{A \gr u} \om_r}{4 \al }   }^2 \mu^{-1} \om_r^{\al-1}  \\
&- \frac{4 \al }{r} \set{ \int_{B_r} u \brac{ A \gr u \cdot x - \frac{\di\pr{A \gr u} \om_r}{4 \al}  } \om_r^{\al-1}}^2 \\
&+ \frac{1}{4 \al r} \brac{\int_{B_r} u \di \pr{A \gr u}  \om_r^{\al}}^2 
- \frac{1}{4 \al r} H(r) \int_{B_r}\brac{  \di\pr{A \gr u} }^2 \mu^{-1} \om_r^{\al+1} \\
&+ \frac {1} r H(r) E_D(r) 
- \frac{1}{r} E_H(r)D(r) .
\end{align*}
Using Cauchy-Schwarz, \eqref{EHBound}, and \eqref{EDBound}, we conclude that there exists $c_0(n, \la) > 0$ so that \eqref{NA'Bound} holds.
\end{proof}

{In the next lemma, we prove a couple of inequalities that will be used below in our monotonicity result and our three-ball inequality.
}

\begin{lemma}[More variable-coefficient estimates]
\label{someBoundsLemma}
Let $A$ be a symmetric, bounded, uniformly elliptic matrix-valued function with Lipschitz continuous coefficients that satisfies \eqref{symm} -- \eqref{LipCond}, and $A(0) = I$.
Assume that $u$ is a solution to $- \di\pr{A \gr u} + W \cdot \gr u + V u= 0$ in $B_R$, where $\norm{W}_{L^\iny(B_R)} \le K$ and $\norm{V}_{L^p\pr{B_R}} \le M$ for some $p \in \brac{n, \iny}$.
Then there exist constants $c_4(n, \la, p), c_5(n, \la, p) > 0$ so that
\begin{align}
\abs{L(r)} &\le c_4 \brac{ M^\ga r^2 + \pr{Kr}^2 + \al \pr{1  + \eta r} }H(r)
+ \frac 1 2 D(r)
\label{LAbound} 
\end{align}
and for {$\be \begin{cases} \text{ arbitrary in } \in [0,1] & \text{ if } p \in (n, \iny] \\ = 1 & \text{ if } p = n \end{cases}$},
\begin{equation}
\label{intdivAGrbound}
\begin{aligned}
\int_{B_r} \abs{\di \pr{A \gr u}}^2 \mu^{-1} \om_r^{\al + 1} 
&\le c_5 \pr{M^{\ga} r^2 }^\be \brac{\pr{M^{\ga} r^2 }^{1 + \frac{1-\be}{1 - \frac n p}}  + \pr{Kr}^2 + \al \pr{1  + \eta r} } H(r) \\
&+ c_5 \brac{\pr{M^{\ga} r^2 }^\be + K^2 r^2}  D(r)
\end{aligned}
\end{equation}
where we introduce $\ga = \frac{2p}{2p-n} \in [1, 2]$.
{If $p = n$, then $\frac{1 - \be}{1 - \frac n p} = 0$.}
\end{lemma}

\begin{proof}
Observe that
\begin{equation*}
\begin{aligned}
\int_{B_r}  A \gr\pr{u \, \om_r^{\frac{\al}2}} \cdot \gr\pr{u \, \om_r^{\frac{\al}2}}
&= \int_{B_r}  A \pr{\gr u \, \om_r -  \al u x  } \cdot \pr{\gr u \, \om_r - \al u x  } \om_r^{\al - 2} \\
&=\int_{B_r}   A \gr u \cdot \gr u \, \om_r^{\al}
- 2 \al \int_{B_r}  u A \gr u \cdot x \om_r^{\al-1} 
+ \al^2 \int_{B_r}  \abs{  u}^2 A x \cdot x \om_r^{\al - 2} \\
&= D(r) - I(r) + \al^2 \int_{B_r}  \abs{  u}^2 A x \cdot x \om_r^{\al - 2},
\end{aligned}
\end{equation*}
where we recall the definitions in \eqref{HDADefn} and \eqref{IADefn}.
An integration by parts followed by the computation from \eqref{HADerivComp} and the notation from \eqref{EHdefn} shows that
\begin{align*}
2(\al - 1) \int_{B_r} \abs{  u}^2 A x \cdot x \, \om_r^{\al - 2}
&= - \int_{B_r} \abs{  u}^2 A x \cdot \gr \om_r^{\al - 1}
= \int_{B_r} \di\pr{ \abs{  u}^2 A x} \om_r^{\al - 1} \\
&= n H(r)
+ E_H(r)
+ \frac 1 \al I(r).
\end{align*}
Putting these together gives
\begin{equation*}
\begin{aligned}
\int_{B_r} A \gr\pr{u \, \om_r^{\frac{\al}2}} \cdot \gr\pr{u \, \om_r^{\frac{\al}2}}
&= D(r) - I(r) 
+ \frac{\al^2}{2(\al - 1)} 2(\al - 1) \int_{B_r}  \abs{  u}^2 A x \cdot x \om_r^{\al - 2} \\
&= D(r) - I(r)  + \frac{\al^2}{2(\al - 1)} \brac{n H(r) + E_H(r) + \frac 1 \al I(r)} \\
&= D(r) + \pr{\frac{\al}{2(\al - 1)}-1} I(r) + \frac{\al^2}{2(\al - 1)} \brac{n H(r) + E_H(r)} \\
&= \frac{\al^2}{2(\al - 1)} \brac{n H(r) + E_H(r)}
+ \frac{\al}{2(\al - 1)} D(r) 
- \frac{\al - 2}{2(\al - 1)} L(r),
\end{aligned}
\end{equation*}
where we have used the relation in \eqref{IADefn}.
Applications of \eqref{EHBound} and \eqref{ellip} show that
\begin{equation}
\label{gradEst}
\int_{B_r} \abs{\gr\pr{u \, \om_r^{\frac{\al}2}}}^2
\le \frac{\al}{2\la (\al - 1)} \brac{ \al \pr{n  + c_1 \eta r} H(r) +  D(r) + \abs{L(r) }}.
\end{equation}

From \eqref{LAdefn} {and the equation for $u$}, we see that
\begin{align*}
L(r) 
&= \int_{B_r} u \di\pr{A \gr u} \om_r^{\al}
= \int_{B_r} V \abs{u}^2 \om_r^{\al}
+ \int_{B_r} u W \cdot \gr u \, \om_r^{\al}.
\end{align*}
{Since $W \in L^\iny$,} an application of Cauchy-Schwarz shows that
\begin{align*}
\int_{B_r} u W \cdot \gr u \, \om_r^{\al}
&\le K r \pr{\int_{B_r} \abs{u}^2 \om_r^{\al -1}}^{\frac 1 2} \pr{\int_{B_r} \abs{\gr u}^2 \om_r^{\al}}^{\frac 1 2} 
\le \frac{K r}{\la} H(r)^{\frac 1 2} D(r)^{\frac 1 2} \\
&\le 2\pr{\frac{Kr}{\la}}^2 H(r) + \frac 1 8 D(r).
\end{align*}
%{
%\begin{align*}
%\int_{B_r} V \abs{u}^2 \om_r^{\al}
%&= \int_{B_r} V \om_r^{1 - \frac n {2p}} \pr{\abs{u}^2  \om_r^{\al-1}}^{1 - \frac n {2p}} \pr{\abs{u}^2  \om_r^{\al}}^{\frac n {2p}}  \\
%&\le \pr{\int_{B_r}\abs{ V}^p}^{\frac 1 p} r^{2 - \frac {n} {p}}   \pr{\int \abs{u}^2  \om_r^{\al-1} }^{1 - \frac n {2p}} \brac{ \int_{B_r} \pr{u  \om_r^{\frac{\al}2}}^{\frac {2n}{n-2}}}^{\frac {n-2} {2p}} \\
%&\le M r^{2 - \frac {n} {p}} \brac{\la^{-1} H(r)}^{1 - \frac n {2p}} \brac{c_s \int_{B_r} \abs{\gr\pr{u \om_r^{\frac{\al}2}}}^2 }^{\frac n {2p}}  \\
%&\le \pr{1 - \frac n {2p}}\brac{\la^{-1} c_s^{\frac n {2p-n}} M^\ga r^2 H(r)}+ \frac {n} {2p} \int_{B_r} \abs{\gr\pr{u \om_r^{\frac{\al}2}}}^2
%\end{align*}
%And then we get
%\begin{align*}
%\abs{L(r)}
%&\le 2\pr{\frac{Kr}{\la}}^2 H(r) + \frac 1 8 D(r)
%+ \pr{1 - \frac n {2p}}\brac{\la^{-1} c_s^{\frac n {2p-n}} M^\ga r^2 H(r)}+ \frac {n} {2p} \int_{B_r} \abs{\gr\pr{u \om_r^{\frac{\al}2}}}^2
%\end{align*}
%so that
%\begin{align*}
%\int_{B_r} \abs{\gr\pr{u \, \om_r^{\frac{\al}2}}}^2
%&\le \frac{\al}{2\la (\al - 1)} \set{ 
%\brac{\al \pr{n  + c_1 \eta r} + 2\pr{\frac{Kr}{\la}}^2 + \pr{1 - \frac n {2p}} \la^{-1} c_s^{\frac n {2p-n}} M^\ga r^2}H(r) 
%+ \frac 5 8 D(r) } \\
%&+ \frac{\al}{2\la (\al - 1)} \set{  \frac {n} {2p} \int_{B_r} \abs{\gr\pr{u \om_r^{\frac{\al}2}}}^2}
%\end{align*}
%}
{If $V \in L^\iny$, then
\begin{align*}
\int_{B_r} V \abs{u}^2 \om_r^{\al}
&\le M r^2 \la^{-1} H(r) 
\end{align*}
so that
\begin{align*}
\abs{L(r) }
&\le \brac{M r^2 \la^{-1} + 2\pr{\frac{Kr}{\la}}^2}H(r) + \frac 1 8 D(r),
\end{align*}
which implies \eqref{LAbound}.
}
On the other hand, {with $V \in L^p$ for $p \in [n, \iny)$, we may interpolate to get}
\begin{align*}
\int_{B_r} V \abs{u}^2 \om_r^{\al}
&= \int_{B_r} V \om_r^{1 - \frac n {2p}} \pr{\abs{u}^2  \om_r^{\al-1}}^{1 - \frac n {2p}} \pr{\abs{u}^2  \om_r^{\al}}^{\frac n {2p}}  \\
&\le \pr{\int_{B_r}\abs{ V}^p}^{\frac 1 p} r^{2 - \frac {n} {p}}   \pr{\int \abs{u}^2  \om_r^{\al-1} }^{1 - \frac n {2p}} \brac{ \int_{B_r} \pr{u  \om_r^{\frac{\al}2}}^{\frac {2n}{n-2}}}^{\frac {n-2} {2p}} \\
&\le M r^{2 - \frac {n} {p}} \brac{\la^{-1} H(r)}^{1 - \frac n {2p}} \brac{c_s \int_{B_r} \abs{\gr\pr{u \om_r^{\frac{\al}2}}}^2 }^{\frac n {2p}} \\
&\le \la^{\frac n {2p} - 1} M r^{2 - \frac {n} {p}} H(r)^{1 - \frac n {2p}} \set{\frac{c_s \al}{2\la (\al - 1)} \brac{ \al \pr{n  + c_1 \eta r} H(r) +  D(r) + \abs{L(r) }}}^{\frac n {2p}} \\
&\le \brac{(2c_s)^{\frac n {2p-n}} \la^{-\ga} M^\ga r^2 H(r)}^{1 - \frac n {2p}} \brac{\al \frac{n  + c_1 \eta r} 2 H(r) +  \frac 1 2 D(r) + \frac 1 2 \abs{L(r) }}^{\frac n {2p}} \\
&\le \pr{1 - \frac n {2p}}\brac{(2c_s)^{\frac n {2p-n}} \la^{-\ga} M^\ga r^2 H(r)} + {\frac n {2p}} \brac{\al \frac{n  + c_1 \eta r} 2 H(r) +  \frac 1 2 D(r) + \frac 1 2 \abs{L(r) }},
\end{align*}
where we have used H\"older's inequality, a Sobolev embedding, \eqref{gradEst}, and Young's inequality.
%{Note that if $p = \iny$, then we interpret $\frac n {p} = 0$ in the above computations.}
That is,
\begin{align*}
\abs{L(r)} 
&\le \brac{(2c_s)^{\frac n {2p-n}} \pr{\frac M \la}^\ga r^2 + 2\pr{\frac{Kr}{\la}}^2 + \al \frac{n  + c_1 \eta r} 4}H(r) 
+ \frac 3 8 D(r)
+ \frac 1 {4}  \abs{L(r)}.
\end{align*}
Absorbing $\abs{L(r)}$ into the left and simplifying shows that \eqref{LAbound} holds.

Now
\begin{align*}
\int_{B_r}\abs{ \di\pr{A \gr u} }^2 \mu^{-1} \om_r^{\al+1}
&= \int_{B_r}\abs{ V u + W \cdot \gr u}^2 \mu^{-1}  \om_r^{\al+1} \\
&\le \frac 2 \la \int_{B_r}\abs{ V u}^2 \om_r^{\al+1}
+ \frac 2 \la \int_{B_r}\abs{W \cdot \gr u}^2 \om_r^{\al+1},
\end{align*}
where
\begin{align*}
\int_{B_r}\abs{W \cdot \gr u}^2 \om_r^{\al+1}
\le \frac{K^2 r^2}{\la} D(r).
\end{align*}
For the other term, we {proceed as we did above with the other term involving $V$; we interpolate, then} use H\"older, Sobolev, and \eqref{gradEst} { to reach our bound.
Since $V \in L^p$ for $p \ge n$, the following interpolation is valid, but wouldn't work for any value of $p < n$.}
We get that
\begin{equation*}
%\label{gradSquareInt}
\begin{aligned}
\int_{B_r}\abs{V u}^2 \om_r^{\al+1}
&= \int_{B_r}\abs{ V}^2 \om_r^{2 - \frac{n}{p}} \pr{\abs{u}^2  \om_r^{\al-1}}^{1 - \frac n {p}} \pr{\abs{u}^2  \om_r^{\al}}^{\frac n {p}} \\
&\le \pr{\int_{B_r}\abs{ V}^p}^{\frac 2 p} r^{4 - \frac{2n}{p}} \pr{\int \abs{u}^2  \om_r^{\al-1} }^{1 - \frac n {p}} \brac{ \int_{B_r} \pr{u  \om_r^{\frac{\al}2}}^{\frac {2n}{n-2}}}^{\frac {n-2} {p}} \\
&\le M^2 r^{4 - \frac{2n}{p}} \brac{\la^{-1} H(r)}^{1 - \frac n p} \brac{c_s \int_{B_r} \abs{\gr\pr{u \om_r^{\frac{\al}2}}}^2}^{\frac n p} \\
&= c_s^{\frac n p} \la^{\frac n p - 1} M^2 r^{4 - \frac{2n}{p}} H(r)^{1 - \frac n p} \brac{  \frac{\al}{2\la (\al - 1)} \brac{ \al \pr{n  + c_1 \eta r} H(r) +  D(r) + \abs{L(r) }}}^{\frac n p} \\
&\le C \pr{M^\ga r^2}^{2 - \frac{n}{p}} H(r)^{1 - \frac n p} \set{ \brac{ M^\ga r^2 + \pr{Kr}^2 + \al \pr{1  + \eta r} }H(r) +  D(r)}^{\frac n p} ,
\end{aligned}
\end{equation*}
where the last {line} uses {the definition of $\ga$ and} the bound for $L(r)$ {in \eqref{LAbound}}.
{If $p = n$, then
\begin{align*}
\int_{B_r}\abs{ \di\pr{A \gr u} }^2 \mu^{-1} \om_r^{\al+1}
&\le C M^\ga r^2 \brac{ M^\ga r^2 + K^2 r^2 + \al \pr{1  + \eta r} }H(r) \\
&+ C\brac{ M^\ga r^2 +  K^2 r^2} D(r),
\end{align*}
which gives \eqref{intdivAGrbound}.
If $p > n$, then for any $\be \in [0,1]$, we write 
$$2 - \frac n p = \be + 1 - \frac n p + 1 - \be = \be + \pr{1 - \frac n p}\pr{1 + \frac{1 -\be}{1 - \frac n p}}$$ 
and then}
\begin{align*}
\int_{B_r}\abs{V u}^2 \om_r^{\al+1}
&\le C  \pr{M^\ga r^2}^{2 - \frac{n}{p}}  H(r)^{1 - \frac n p} \set{ \brac{ M^\ga r^2 + \pr{Kr}^2 + \al \pr{1  + \eta r} }H(r) +  D(r)}^{\frac n p} \\
&= C \pr{M^{\ga} r^2 }^\be \brac{\pr{M^{\ga} r^2 }^{1 + \frac{1-\be}{1 - \frac n p}}  H(r)}^{1 - \frac n p} \\
&\times\set{ \brac{ M^\ga r^2 + \pr{Kr}^2 + \al \pr{1  + \eta r} }H(r) +  D(r)}^{\frac n p} \\
&\le C \pr{M^{\ga} r^2 }^\be \set{\brac{\pr{M^{\ga} r^2 }^{1 + \frac{1-\be}{1 - \frac n p}}  + \pr{Kr}^2 + \al \pr{1  + \eta r} }H(r) +  D(r)}.
\end{align*}
In combination with the bound on the gradient term, we reach the conclusion described by \eqref{intdivAGrbound}.
\end{proof}

If $u$ is a solution to a Schr\"odinger equation, we get the following monotonicity result.

\begin{cor}[Variable-coefficient monotonicity result]
\label{monoCorA}
Let $A$ be a symmetric, bounded, uniformly elliptic matrix-valued function with Lipschitz continuous coefficients that satisfies \eqref{symm} -- \eqref{LipCond}, and $A(0) = I$.
Assume that $u$ is a solution to $- \di\pr{A \gr u} + W \cdot \gr u + V u= 0$ in $B_R$, where $\norm{W}_{L^\iny(B_R)} \le K$ and $\norm{V}_{L^p\pr{B_R}} \le M$ for some $p \in \brac{n, \iny}$.
Then there exist constants, $C_1(n, \la, p), C_2(n, \la, p)$ so that with $N(r)$ as in \eqref{NADefn} and {$\be \begin{cases} \text{ arbitrary in } \in [0,1] & \text{ if } p \in (n, \iny] \\ = 1 & \text{ if } p = n \end{cases}$}, it holds that
\begin{equation*}
\widetilde N(r) := \set{ N(r) + \frac{C_2}{\al} \pr{M^{\ga} r^2 }^\be \brac{\pr{M^{\ga} r^2 }^{1 + \frac{1-\be}{1 - \frac n p}}  + \pr{Kr}^2 + \al\pr{1 + \eta r}} } e^{c_0 \eta r + \frac{C_1} \al \brac{ \pr{M^{\ga} r^2 }^\be + K^2 r^2 } },
\end{equation*}
is a non-decreasing function of $r$.
{Here $\al \ge 2$, $\ga = \frac{2p}{2p-n}$, and $\frac{1 - \be}{1 - \frac n p} = 0$ when $p = n$.}
\end{cor}

\begin{proof}
Since$- \di\pr{A \gr u} + W \cdot \gr u + V u= 0$ in $B_R$, then we may use \eqref{intdivAGrbound} and \eqref{NA'Bound} from Lemma \ref{monoLemmaA} to get
\begin{align*}
N'(r) + \brac{c_0\eta +  c_5 \frac{ \pr{M^{\ga} r^2 }^\be + \pr{Kr}^2 }{4\al r} } N(r)
\ge - c_5 \frac{ \pr{M^{\ga} r^2 }^\be \brac{\pr{M^{\ga} r^2 }^{1 + \frac{1-\be}{1 - \frac n p}}  + \pr{Kr}^2 + \al \pr{1  + \eta r} } }{4 \al r }.
\end{align*}
The result follows from an integration.
\end{proof}

Using the monotonicity result above and choosing $\al \gg 1$ leads to the following three-balls inequality.

\begin{prop}[Variable-coefficient three-balls inequality]
\label{threeBallsA}
Let $A : B_R \to \mathbb{M}_n$ be a symmetric, bounded, uniformly elliptic matrix-valued function with Lipschitz continuous coefficients that satisfies \eqref{symm} -- \eqref{LipCond}, and $A(0) = I$.
Let $u$ be a solution to $- \di\pr{A \gr u} + W \cdot \gr u + V u= 0$ in $B_R$, where $\norm{W}_{L^\iny(B_R)} \le K$ and $\norm{V}_{L^p\pr{B_R}} \le M$ for some $p \in \brac{n, \iny}$.
Then for any $0 < r_1 < r_2 < \si r_2 < r_3 < R$, {where $\si \in \pr{1, \frac {r_3}{r_2}}$ is a constant,} it holds that
\begin{align*}
\norm{u}_{r_2}
\le \frac{e^{C \eta R}}{\la^2} \exp\set{ \brac{\pr{M^{\ga} R^2 }^\be + K^2 R^2 + C}\brac{C + 2  \log\pr{\frac{r_3}{\si r_2}} + \log\pr{\frac{\si^2}{\si^2 -1}}} } 
\norm{u}_{r_1}^{\kappa} \norm{u}_{r_3}^{1 - \kappa},
\end{align*}
where $\norm{u}_r := \norm{u}_{L^2(B_r)}$, $C(n, \la, p) > 0$, $\be = \frac{2p - n}{3p - 2n}$, $\ga = \frac{2p}{2p-n}$, and 
$$\kappa = \frac{ \log r_3 - \log\pr{\si r_2}}{\log r_3 - \log\pr{\si r_2} + 3 e^{c_0 \eta R + C_1} \brac{\log\pr{\si r_2} - \log r_1}},$$ 
with $c_0(n, \la) > 0$ from Lemma \ref{monoLemmaA} and $C_1(n, \la, p) > 0$ from Corollary \ref{monoCorA}. 
\end{prop}

\begin{remark}
If $A \equiv I$, then the above result holds with $\la = 1$ and $\eta = 0$. 
\end{remark}

\begin{proof}
From \eqref{HADeriv}, we have
\begin{align*}
\frac{H'(r) }{H(r)}
&= \frac{2 (\al -1) + n}{r} + \frac {N(r)} {\al r} + \frac{1}{\al r}\frac{L(r)}{H(r)} + \frac{E_H(r)}{r H(r)},
\end{align*}
where we recall \eqref{NADefn}. 
{Set $\be_1 = 1 + \frac{1-\be}{1 - \frac n p} \ge 1$, where $\be \in (0, 1]$ is to be determined for $p \ne n$, and $\be_1 = \be = 1$ when $p = n$.}
%If we choose $\be = \frac{2p - n}{3p - 2n}$, then $1 + \frac{1-\be}{1 - \frac n p} = \frac{4p - 2n}{3p- 2n} = 2 \be$.
Applications of \eqref{EHBound} and \eqref{LAbound} %with this choice of $\be$ 
show that
\begin{equation}
\label{logH'ALow}
\begin{aligned}
\frac{H'(r) }{H(r)}
&\ge \frac{2 \al + n - 2 - c_4}{r} 
+ \frac {N(r) + \frac{C_2}{\al} \pr{M^{\ga} r^2 }^\be \brac{\pr{M^{\ga} r^2 }^{{\be_1}}  + \pr{Kr}^2 + \al\pr{1 + \eta r}}} {2\al r} \\
&- \frac {C_2  \pr{M^{\ga} r^2 }^\be \brac{\pr{M^{\ga} r^2 }^{{\be_1}}  + \pr{Kr}^2 + \al\pr{1 + \eta r}}} {2\al^2 r}
- \frac{c_4 \brac{ M^\ga r^2 + \pr{Kr}^2} }{\al r} 
- c_6 \eta
\end{aligned}
\end{equation}
and
\begin{equation}
\label{logH'AUp}
\begin{aligned}
\frac{H'(r) }{H(r)}
&\le \frac{2 \al + n - 2 + c_4}{r} + 3 \frac {N(r) + \frac{C_2}{\al} \pr{M^{\ga} r^2 }^\be \brac{\pr{M^{\ga} r^2 }^{{\be_1}}  + \pr{Kr}^2 + \al\pr{1 + \eta r}}} {2\al r}
 \\
&+ \frac{c_4 \brac{ M^\ga r^2 + \pr{Kr}^2 } }{\al r} + c_6 \eta ,
\end{aligned}
\end{equation}
where $c_6 = c_1 + c_4$.
If we assume that 
\begin{equation}
\label{alADefn}
\al \ge {2 + } \pr{M^{\ga} R^2 }^\be + K^2 R^2,
\end{equation}
then an application of Corollary \ref{monoCorA} implies that for all $r \in \brac{\si r_2, r_3} $,
\begin{equation}
\label{tildeNALow}
\begin{aligned}
N(r) + \frac{C_2}{\al} \pr{M^{\ga} r^2 }^\be \brac{\pr{M^{\ga} r^2 }^{{\be_1}}  + \pr{Kr}^2 + \al\pr{1 + \eta r}}
&=\widetilde N(r) e^{- c_0 \eta r - \frac{C_1}{\al} \brac{ \pr{M^{\ga} r^2 }^\be + K^2 r^2 } } \\
&\ge \widetilde N(\si r_2) e^{-C_3},
\end{aligned}
\end{equation}
where we introduce $C_3 := c_0 \eta R + C_1$.
Similarly, for all $r \in \brac{r_1, \si r_2}$, we have
\begin{align}
\label{tildeNAUp}
N(r) + \frac{C_2}{\al} \pr{M^{\ga} r^2 }^\be \brac{\pr{M^{\ga} r^2 }^{{\be_1}}  + \pr{Kr}^2 + \al\pr{1 + \eta r}}
&\le \widetilde N(\si r_2).
\end{align}
{If we choose $\be = \frac{2p - n}{3p - 2n}$, then $\be_1 = 1 + \frac{1-\be}{1 - \frac n p} = \frac{4p - 2n}{3p- 2n} = 2 \be$ for $p > n$ and $\be_1 = 1 \le 2 \be$ for $p = n$.
As we'll see below, this choice of $\be$ will make the choice of $\al$ in \eqref{alADefn} optimal.}
Using \eqref{tildeNALow} and \eqref{alADefn} in \eqref{logH'ALow} shows that for all $r \in \brac{\si r_2, r_3}$,
\begin{equation}
\label{logH'ALow2}
\begin{aligned}
\frac{H'(r) }{H(r)}
&\ge \frac{2 \al + n - 2 }{r} 
+ \frac {\widetilde{N}(r)} {2 e^{C_3} \al r} 
- C\brac{\frac 1  r +  \frac {\pr{M^{\ga} r^2 }^{\be}  } {r} + \eta}.
\end{aligned}
\end{equation}
Similarly, using \eqref{tildeNAUp} and \eqref{alADefn} in \eqref{logH'AUp} shows that for all $r \in \brac{r_1, \si r_2}$,
\begin{equation}
\label{logH'AUp2}
\begin{aligned}
\frac{H'(r) }{H(r)}
&\le \frac{2 \al + n - 2}{r} 
+ \frac {3\widetilde N(\si r_2)} {2\al r}
+ C\brac{\frac 1 r +  \frac {\pr{M^{\ga} r^2 }^{\be}  } {r} + \eta}.
\end{aligned}
\end{equation}
Integrating \eqref{logH'ALow2} from $\si r_2 < r_3 < R$ gives
\begin{align*}
\log\brac{\frac{H(r_3)}{H(\si r_2)}}
&= \int_{\si r_2}^{r_3} \frac{H'(r) }{H(r)} dr 
\ge  \int_{\si r_2}^{r_3} \set{\frac{2 \al + n - 2}{r} 
+ \frac {\widetilde{N}(r)} {2 e^{C_3} \al r} 
- C\brac{\log r +  \frac {\pr{M^{\ga} r^2 }^{\be}  } {r} + \eta}} dr
  \\
&\ge \frac{2 \al + n -2 + \frac {\widetilde N(\si  r_2)}{\al} }{2 e^{C_3}} \log\pr{\frac{r_3}{\si r_2}} 
 - C \brac{  \log\pr{\frac{r_3}{\si r_2}}  + \pr{M^{\ga} R^2 }^{\be} + \eta R}.
\end{align*}
Integrating \eqref{logH'AUp2} from $\si r_2 < r_3 < R$ gives
\begin{align*}
\log\brac{\frac{H(\si r_2)}{H(r_1)}}
&\le \int_{r_1}^{\si r_2} \set{\frac{2 \al + n - 2}{r} 
+ \frac {3\widetilde N(\si r_2)} {2\al r}
+ C\brac{\log r +  \frac {\pr{M^{\ga} r^2 }^{\be}  } {r} + \eta}} dr 
 \\
&\le \frac 3 2 \brac{2 \al + n -2 + \frac {\widetilde N(\si r_2)} {\al}} \log\pr{\frac{\si r_2}{r_1}}
+ C \brac{  \log\pr{\frac{\si r_2}{r_1}} 
+ \pr{M^{\ga} R^2 }^{\be}
+ \eta R}.
\end{align*}
Combining these inequalities shows that
\begin{align*}
\frac{\log\brac{\frac{H(\si r_2)}{H(r_1)}} - C \brac{  \log\pr{\frac{\si r_2}{r_1}} + \pr{M^{\ga} R^2 }^{\be} + \eta R}}{3 e^{C_3} \log\pr{\frac{\si r_2}{r_1}}}
&\le \frac {2 \al + n -2 + \frac {\widetilde N(\si r_2)} {\al}}{2e^{C_3}}  \\
&\le \frac{\log\brac{\frac{H(r_3)}{H(\si r_2)}} + C \brac{  \log\pr{\frac{r_3}{\si r_2}}  + \pr{M^{\ga} R^2 }^{\be} + \eta R}}{\log\pr{\frac{r_3}{\si r_2}} }.
\end{align*}

With $h(r)$ as in \eqref{hDefn}, observe that $H(r) \le \la^{-1} r^{2(\al-1)} h(r)$, while for any $0 < r < \rho \le 1$, $\disp h(r) \le \la^{-1} \frac{H(\rho)}{\pr{\rho^2 - r^2}^{\al-1}}$.
In particular, 
$$H(\si r_2) \ge \la \brac{\pr{\si^2 - 1} r_2^2}^{\al -1} h(r_2) = \la \pr{\frac {\si^2 -1}{\si^2}}^{\al -1} \pr{\si r_2}^{2(\al-1)} h(r_2).$$
Therefore,
\begin{align*}
\log\brac{\frac{H(r_3)}{H(\si r_2)}}
&\le \log\brac{\frac{\la^{-2} r_3^{2(\al-1)} h(r_3)}{\pr{\frac {\si^2 -1}{\si^2}}^{\al -1} \pr{\si r_2}^{2(\al-1)}  h(r_2)}}  \\
&= \log\brac{\frac{h(r_3)}{h(r_2)}} + (\al-1) \brac{2\log\pr{\frac{r_3}{\si r_2}} + \log\pr{\frac{\si^2}{\si^2 -1}}} - 2 \log \la
\end{align*}
and
\begin{align*}
\log\brac{\frac{H(\si r_2)}{H(r_1)}}
&\ge \log\brac{\frac{\la^2 \pr{\frac{\si^2 -1}{\si^2}}^{\al -1} \pr{\si r_2}^{2(\al-1)} h(r_2)}{r_1^{2(\al-1)} h(r_1)}} \\
&= \log\brac{\frac{h(r_2)}{h(r_1)}} + (\al-1) \brac{2\log\pr{\frac{\si r_2}{r_1}} - \log\pr{\frac {\si^2}{\si^2 -1}}} + 2 \log \la.
\end{align*}
Substituting these bounds into the expression above gives
\begin{align*}
&\frac{\log\brac{\frac{h(r_2)}{h(r_1)}} + (2\al-2 - C)\log\pr{\frac{\si r_2}{r_1}}  - (\al-1) \log\pr{\frac {\si^2}{\si^2 -1}} + 2 \log \la - C \brac{ \pr{M^{\ga} R^2 }^{\be} + \eta R}}{3 e^{C_3} \log\pr{\frac{\si r_2}{r_1}}} \\
&\le \frac{\log\brac{\frac{h(r_3)}{h(r_2)}} + (2\al-2 + C)\log\pr{\frac{r_3}{\si r_2}}  + (\al-1) \log\pr{\frac{\si^2}{\si^2 -1}} - 2 \log \la + C \brac{  \pr{M^{\ga} R^2 }^{\be} + \eta R}}{\log\pr{\frac{r_3}{\si r_2}} }.
\end{align*}
Set $\rho = 3 e^{C_3} \log\pr{\frac{\si r_2}{r_1}}$ and $\tau = \log\pr{\frac{r_3}{\si r_2}} $, then simplify to get
\begin{align*}
\log h(r_2) 
&\le \frac{\rho}{\rho + \tau}  \log h(r_3) + \frac{\tau}{\rho + \tau} \log h(r_1) 
+ \frac{2\rho \tau}{\rho + \tau} \pr{C + \al} \\
&+ \al \log\pr{\frac{\si^2}{\si^2 -1}} - 2 \log \la + C \brac{ \pr{M^{\ga} R^2 }^{\be} + \eta R} \\
&\le \log h(r_3)^{\frac{\rho}{\rho + \tau}  } + \log h(r_1)^{\frac{\tau}{\rho + \tau}}
+ \al\brac{C + 2 \tau + \log\pr{\frac{\si^2}{\si^2 -1}}} 
+ 2 C \tau - 2 \log \la + C  \eta R,
\end{align*}
where we have used \eqref{alADefn}.
Exponentiating then shows that
\begin{align*}
h(r_2)
&\le \exp\set{ \pr{\al + C}\brac{C + 2  \log\pr{\frac{r_3}{\si r_2}} + \log\pr{\frac{\si^2}{\si^2 -1}}} - 2 \log \la + C  \eta R} 
h(r_1)^{\kappa} 
h(r_3)^{1 - \kappa},
\end{align*}
where $\disp \kappa = \frac{ \log r_3 - \log\pr{\si r_2}}{\log r_3 - \log\pr{\si r_2} + 3 e^{C_3} \log\pr{\si r_2} - 3 e^{C_3} \log r_1} $.
With {$\al = 2 + \pr{M^\ga R^2}^\be + K^2 R^2$, \eqref{alADefn} holds} and we reach the conclusion.
\end{proof}

\section{Proofs of Main Theorems}
\label{S:Proofs}

We now have everything we need to prove Theorems \ref{OofV1}, \ref{UCatIny}, \ref{OofV}, and \ref{UCatInyAWV}.
{The first three proofs follow from applications of the three-ball inequality in Proposition of \ref{threeBallsV} or Proposition \ref{threeBallsA}.}

{We begin with the proof of Theorem \ref{OofV1}, an order of vanishing estimate for solutions to $-\LP u + V u = 0$ when $V \in L^\iny$.}

\begin{proof}[Proof of Theorem \ref{OofV1}]
The boundedness condition \eqref{uBound} implies that $\norm{u}_{L^2(B_{10})} \le C_0 \abs{B_{10}}^{\frac 1 2}$. 
The normalization condition \eqref{uNorm} followed by an application of Proposition \ref{threeBallsV} with $\si = 2$, $r_1 =r$, $r_2 = 1$, and $R = r_3 = 10$ shows that
\begin{align*}
1
\le \norm{u}_{L^2(B_{1})} 
&\le \exp\set{ C \brac{{1 +} \pr{100 M}^{\frac{2}{3}}} } \norm{u}_{L^2(B_{r})} ^{\kappa} \norm{u}_{L^2(B_{10})} ^{1 - \kappa}  \\
&\le \exp\pr{ \tilde C M^{\frac{2}{3}} { + C} + \ln C_0 + \tfrac 1 2 \ln \abs{B_{10}}} \norm{u}_{L^2(B_{r})} ^{\kappa},
\end{align*}
where $\tilde C = 100^{\frac 2 3} C$ and $\disp \kappa^{-1} = \frac{\log 10 - \log r}{\log 5}$.
After rearranging, we see that
\begin{align*}
 \norm{u}_{L^2(B_{r})}
 &\ge \exp\brac{\frac{\log r - \log 10}{\log 5} \pr{\tilde C M^{\frac{2}{3}} { + C} + \ln C_0 + \tfrac 1 2 \ln \abs{B_{10}}}}
\end{align*}
If $r \le \tfrac 1 {10}$, then further simplifications lead to \eqref{OofVResult1}.
\end{proof}

{Next, we prove Theorem \ref{UCatIny}, a Landis-type result for solutions to $-\LP u + V u = 0$ in $\R^n$ when $V \in L^\iny$.}

\begin{proof}[Proof of Theorem \ref{UCatIny}]
Let $x_0 \in \R^n$ with $\abs{x_0} = R$.
Define the shifted functions $u_0(x) = u(x + x_0)$ and $V_0(x) = V(x + x_0)$ so that $\norm{V_0}_{L^\iny(B_1)} \le 1$ and $u_0$ solves $- \LP u_0 + V_0 u_0 = 0$ in $B_{3R}$.
The boundedness condition \eqref{uBound1} implies that $\norm{u_0}_{L^2(B_{3R})} \le \abs{B_{3R}}^{\frac 1 2} \exp\brac{C_0 \pr{4 R}^{\frac{4}{3}}}$, while the normalization condition \eqref{uNorm1} shows that $\norm{u_0}_{L^2\pr{B_{R+1}}} \ge 1$.
An application of Proposition \ref{threeBallsV} with $r_1 = 1$, $r_2 = R+1$, and $r_3 = 3R$ shows that
\begin{align*}
1 
\le \norm{u_0}_{L^2(B_{R+1})} 
&\le \exp\set{ C \brac{{ 1+ } (3R)^{\frac{4}{3}} }} \norm{u_0}_{L^2(B_{1})} ^{\kappa} \norm{u_0}_{L^2(B_{3R})} ^{1 - \kappa} \\
&\le \exp\pr{ \tilde C R^{\frac{4}{3}} { + C }} \norm{u_0}_{L^2(B_{1})} ^{\kappa} \pr{\abs{B_{3R}}^{\frac 1 2} \exp\brac{C_0 \pr{4 R}^{\frac{4}{3}}}}^{1 - \kappa} ,
\end{align*}
where $\tilde C = 3^{\frac 4 3}C$ and $\disp \kappa^{-1} = \frac{\log R + \log 3}{\log \pr{\frac{3}{2 + 2R^{-1}}}}$.
Simplifying this expression shows that
\begin{align*}
\norm{u_0}_{L^2(B_{1})}
&\ge \exp\brac{- \frac{\tilde C R^{\frac{4}{3}} + C_0 \pr{4 R}^{\frac{4}{3}} { + C} + \tfrac 1 2 \log \abs{B_{3R}}}{\kappa} } .
\end{align*}
Since $\kappa^{-1} \le \frac{2}{\log\pr{{\tfrac 9 8}}} \log R$ whenever $R \ge 3$, then \eqref{UCatInyResult} holds, as required.
\end{proof}

{Our third proof, Theorem \ref{OofV}, establishes an order of vanishing estimate for solutions to $-\di \pr{A \gr u} + W \cdot \gr u + V u = 0$, where $W \in L^\iny$ and $V \in L^p$ for any $p \ge n$.}

\begin{proof}[Proof of Theorem \ref{OofV}]
The boundedness condition \eqref{uBound} implies that $\norm{u}_{L^2(B_{10})} \le C_0 \abs{B_{10}}^{\frac 1 2}$. 
The normalization condition \eqref{uNorm} followed by an application of Proposition \ref{threeBallsA} with $\si = 2$, $r_1 =r$, $r_2 = 1$, and $R = r_3 = 10$ shows that
\begin{align*}
1
\le \norm{u}_{L^2(B_{1})} 
&\le \frac{e^{10C \eta}}{\la^2} \exp\set{ \brac{\pr{100 M^{\ga}}^\be + 100 K^2 + C}\brac{C + \log \pr{\tfrac{100} 3 }}  }
 \norm{u}_{L^2(B_{r})} ^{\kappa} \norm{u}_{L^2(B_{10})} ^{1 - \kappa} \\
 &\le \exp\brac{ \tilde C \pr{K^2 + M^{\frac{2p}{3p - 2n}} + \eta + 1}+ \ln C_0 + \tfrac 1 2 \ln \abs{B_{10}}} \norm{u}_{L^2(B_{r})} ^{\kappa} ,
\end{align*}
where $\tilde C(n, \la, p) > 0$ and $\disp \kappa^{-1} = 1 - \frac{3 e^{10 c_0 \eta + C_1} }{ \log 5}\log \pr{\frac r 2}$.
After rearranging, we see that
\begin{align*}
\norm{u}_{L^2(B_{r})}
&\ge \exp\set{\brac{ \log \pr{\frac r 2} - \frac{ \log 5}{3 e^{10 c_0 \eta + C_1} }} \tilde C_0 \pr{K^2 + M^{\frac{2p}{3p - 2n}} + \eta + 1}e^{10 c_0 \eta}} ,
\end{align*}
where $\tilde C_0(n, \la, p, C_0) > 0$.
Choose $r_0(n, \la, \eta, p) > 0$ so that $\disp \log \pr{\frac r 2} - \frac{ \log 5}{3 e^{10 c_0 \eta + C_1} } \gtrsim \log r$ whenever $r \le r_0$ and then \eqref{OofVResult} follows.
\end{proof}

{Finally, we come to the proof of Theorem \ref{UCatInyAWV}.
This theorem establishes a Landis-type result for solutions to $-\di \pr{A \gr u} + W \cdot \gr u + V u = 0$ in $\R^n$, where $W \in L^\iny$ and $V \in L^p$ for any $p \ge n$.
To prove this theorem, we need to assume that $\abs{\gr A}$ exhibits sufficient decay at infinity.
Accordingly, the proof requires an iterative argument with repeated applications of the three-ball inequality in Proposition \ref{threeBallsA}.
The first application of Proposition \ref{threeBallsA} gives a super-exponential rate of decay that comes from the contribution of the Lipschitz constant.
However, as we move away from the origin, the decay of $\abs{\gr A}$ reduces the contribution of the Lipschitz constant.
In fact, each subsequent application of the three-ball inequality gives a rate of decay that depends exponentially on the current point and super-exponentially on the starting point.
Therefore, to prove our desired rate of decay, we carry out an iteration scheme with a fixed starting point.
%Since the first step always contributes a super-exponential term, we fix a starting point, and iterate our arguments to reach the point $x_0 \in \R^n$ at which we want the rate of decay.
%By carefully iterating the argument an appropriate number of times, this power becomes smaller, and we can reach our desired bound.
}

\begin{proof}[Proof of Theorem \ref{UCatInyAWV}]
We show that there exists a positive constants $\tilde c_1(n, \la, \eta)$ {and a distance} $\bar{R}_1(n, \la, \eta, \eps, p, C_0, \de)$ so that the following holds:
If $R_1 \ge \bar{R}_1$, then for any $j \in \N$ and any $x_j \in \R^n$ with $\abs{x_j} = R_1^{(1+\de)^j}$, it holds that
\begin{align*}
\norm{u}_{L^2\pr{B_{1}(x_j)}}
&\ge \exp\brac{- \pr{e^{\tilde c_1 R_1} +1} \abs{x_j}^{2\pr{1+\de}}}.
\end{align*}
We proceed by induction.

For some $R_1 \gg 1$ to be specified below, choose $x_1 \in \R^n$ with $\abs{x_1} = R_1$.
To apply our three-ball inequality, we need $A(0) = I$, so we first perform a change of variables.
Since $A(x_1)$ is symmetric and positive definite, then with $S_1 = A(x_1)^{\frac 1 2}$, define $A_1(x) = S_1^{-1} A(x_1 + S_1 x) S_1^{-1}$. 
Then $A_1$ satisfies \eqref{symm}, \eqref{Abound} and \eqref{ellip} with $\la^2$, \eqref{LipCond} with $\la^{-\frac 1 2} \eta$, and $A_1(0) = I$. 
Moreover, with $u_1(x) = u(x_1 + S_1 x)$, $W_1(x) = S_1^{-1} W(x_1 + S_1 x)$, $V_1(x) = V(x_1 + S_1 x)$, and $R = 3 \la^{-\frac 1 2} R_1$, it holds that
$$- \di \pr{A_1 \gr u_1} + W_1 \cdot \gr u_1 + V_1 u_1 = 0 \quad \text{ in } B_R,$$
where $\norm{V_1}_{L^p(B_{R})} \le \la^{- \frac 1 {2p}}$ and $\norm{W_1}_{L^\iny(B_{R})} \le \la^{-\frac 1 2}$.  
Since $\abs{x_1 + S_1 x} \le \frac 4 \la R_1$ in $B_R$, then the boundedness condition \eqref{uBound2} implies that 
\begin{equation}
\label{u1Bound}
\norm{u_1}_{L^2(B_{R})} 
\le \abs{B_{R}}^{\frac 1 2} \exp\brac{C_0 \pr{ \tfrac {4} \la}^{2} R_1^2}
\le \exp\brac{C_0 \pr{ \tfrac {5} \la}^{2} R_1^2},
\end{equation}
where we have assumed that $R_1$ is large enough for $\log \abs{B_{1}} + n \log \pr{\frac{3R_1} \la} \le 2C_0 \pr{ \frac {3} \la}^{2} R_1^2$ to hold.
With $r_2 = \la^{-\frac 1 2}\pr{ R_1+1}$, we have $B_1(0) \su \set{x_1 + S_1 x : x \in B_{r_2}} $, so the normalization condition described by \eqref{uNorm2} shows that $\norm{u_1}_{L^2\pr{B_{r_2}}} \ge {\la^{n/2}}$.
An application of Proposition \ref{threeBallsA} with $r_1 = \la^{\frac 1 2}$, $r_2 = \la^{-\frac 1 2} (R_1 +1)$, $\si = 2$, and $r_3 =R = 3 \la^{-\frac 1 2} R_1$ shows that
\begin{align*}
{\la^{n/2}}
&\le \norm{u_1}_{L^2\pr{B_{r_2}}} \\
&\le \frac{e^{{\bar C} R_1\eta }}{\la^4} \exp\set{ \brac{ \pr{\tfrac{3}{\la}}^2 R_1^2 + \pr{ \tfrac{3 }{\la^{{\nu}}} }^{2\be} R_1^{2\be}+ C}\brac{C + \log \pr{\tfrac 4 3} + 2 \log\pr{\tfrac 3 2 \pr{1 + R_1^{-1}}^{-1}}}} \\
&\times \norm{u_1}_{L^2\pr{B_{r_1}}}^{\kappa_1}  \exp\brac{C_0 \pr{ \tfrac {5} \la}^{2} R_1^2},
\end{align*}
where {$\bar C = \frac{3C} \la$, $\nu = \frac 1 2+ \frac 1 {4p-2n}$,} $\be = \frac{2p - n}{3p - 2n} \le 1$, $\disp \kappa_1^{-1} = 1 + 3 e^{{\tilde{c}_0} \eta R_1 + C_1} \frac{\log \pr{\frac 2 \la \pr{R_1 + 1}}}{\log\pr{\tfrac 3 2 \pr{1 + R_1^{-1}}^{-1}}}$ {with $\tilde{c}_0 = \frac{3 c_0}{\la}$}, and we have substituted the bound from \eqref{u1Bound}. 
If $R_1 \ge \eta$, then with $\tilde C(n, \la, p, C_0) > 0$, we see that
\begin{align*}
1
&\le \exp\pr{\tilde C R_1^2 } \norm{u_1}_{L^2\pr{B_{r_1}}}^{\kappa_1}
\end{align*}
or
\begin{align*}
\norm{u_1}_{L^2\pr{B_{r_1}}}
&\ge \exp\set{- \tilde C\brac{3 e^{{\tilde c_0} \eta R_1 + C_1} \frac{\log \pr{\frac 2 \la \pr{R_1 + 1}}}{\log\pr{\tfrac 3 2 \pr{1 + R_1^{-1}}^{-1}}} + 1 } R_1^2 } \\
&\ge \exp\set{-  e^{\tilde c_1 R_1}  {\frac{3 \tilde C e^{C_1} \log \pr{\frac{3}{\la}R_1}}{\log\pr{\tfrac 3 2 \pr{1 + R_1^{-1}}^{-1}}}} R_1^2 } ,
%\ge \exp\pr{-\tilde C_1 e^{\tilde c_1 R_1} R_1^2 \log R_1 },
\end{align*}
where $\tilde c_1(n, \la, \eta) = {\tilde c_0} \eta$.
If we choose $R_1(n, \la, p, C_0, \de) \gg 1$ so that $\frac{3 \tilde C e^{C_1}  \log \pr{\frac{3}{\la}R_1}}{\log\pr{\tfrac 3 2 \pr{1 + R_1^{-1}}^{-1}}} \le R_1^{2 \de}$, then 
\begin{align*}
\norm{u}_{L^2\pr{B_{1}(x_1)}}
&\ge \exp\brac{-e^{\tilde c_1 R_1}  \abs{x_1}^{2\pr{1+\de}}},
\end{align*}
which implies the base case.

Define $\bar R_1$ to be the smallest number $R_1 > 0$ so that with the constants as above, the following conditions hold:
\begin{align}
& 2C_0 \pr{ \tfrac {3} \la}^{2} R_1^2 \ge \log \abs{B_{1}} + n \log \pr{\tfrac{3} \la R_1} 
\label{R1Cond*} \\
& {R}_1 \ge \max\set{\eta, \frac{12}{\la}, \brac{\frac{{\tilde{c}_0} \eta}{\log\pr{\frac{10}{3}}}}^{\frac 1{\eps - \de}}}
\label{R1Cond0} \\
&\frac{3 \tilde C e^{C_1} \log \pr{\frac{3}{\la}R_1}}{\log\pr{\frac 3 2 \pr{1 + R_1^{-1}}^{-1}  }} \le R_1^{2 \de} 
\label{R1Cond1} \\
&1 + \frac{10 e^{C_1} \log (\la^{-1}R_1)}{\log \pr{1 + \frac{\la}{4} R_{1}^{-\frac{\de}{1 + \de}}}}
\le R_1^{\frac 3 2 \de} 
\label{R1Cond2} \\
&\tilde C  \log\pr{\tfrac 3 \la R_{1}^{\de}} + \tilde C +1 
\le R_1^{\frac \de 2}.
\label{R1Cond3}
\end{align}
Note that conditions \eqref{R1Cond*}, \eqref{R1Cond0} (first part), and \eqref{R1Cond1} were used already in the base case.
We remark that $\bar{R}_1$ depends on $n$, $\la$, $\eta$, $\eps$, $p$, $C_0$, and $\de$.

We now carry out the inductive step.
Assume that $R_{k-1} \ge \bar{R}_1$ has been defined and that for any $x_{k-1} \in \R^n$ with $\abs{x_{k-1}} = R_{k-1}$, it holds that
\begin{align*}
\norm{u}_{L^2\pr{B_{1}(x_{k-1})}}
&\ge \exp\brac{- \pr{e^{\tilde c_1 R_1} + 1}\abs{x_{k-1}}^{2\pr{1+\de}}}.
\end{align*}
Set $R_k = R_{k-1}^{1 + \de}$ and choose $x_k \in \R^n$ with $\abs{x_k} = R_k$.
%Since $A(x_k)$ is symmetric and positive definite, then w
With $S_k = A(x_k)^{\frac 1 2}$, define $A_k(x) = S_k^{-1} A(x_k + S_k x) S_k^{-1}$ and observe that $A_k$ satisfies \eqref{symm}, \eqref{Abound} and \eqref{ellip} with $\la^2$, and $A_k(0) = I$. 
Moreover, with $u_k(x) = u(x_k + S_k x)$, $W_k(x) = S_k^{-1} W(x_k + S_k x)$, and $V_k(x) = V(x_k + S_k x)$, we have
$$- \di \pr{A_k \gr u_k} + W_k \cdot \gr u_k + V_k u_k = 0,$$
where $\norm{V_k}_{L^p} \le \la^{- \frac 1 {2p}}$ and $\norm{W_k}_{L^\iny} \le \la^{-\frac 1 2}$.

Define $\rho > 0$ as
$$\rho = \inf\set{r > 0 : \set{x_k + S_k x : x \in B_r(0)} \cap \del B_{R_{k-1}}(0) \ne \emptyset}.$$
The ellipticity condition shows that $\rho = \la^a \pr{R_k - R_{k-1}}$ for some $a \in\brac{-\frac 1 2, \frac 1 2}$.
Then set $R > 0$ to be
$$R = \sup \set{r > 0 : \set{x_k + S_k x : x \in B_r(0)} \cap B_{R_{k-1}/2}(0) = \emptyset}.$$
In this case, the ellipticity condition shows that $R = \la^a \pr{R_k - \frac 1 2 R_{k-1}}$ for some $a \in \brac{-\frac 1 2, \frac 1 2}$.
However, we may also deduce that $R - \rho \ge \frac 1 2 \sqrt \la R_{k-1}$.
Since $R \le \la^{-\frac 1 2} R_k$, then in $B_R$, we have $\abs{x_k + S_k x} \le R_k + \la^{-\frac 1 2} R \le 2 \la^{-1} R_k$.
Therefore, \eqref{LipCond} holds in $B_R$ with $2^{1 + \eps} \la^{-\frac 1 2} \eta R_{k-1}^{-1 - \eps}$.
The boundedness condition \eqref{uBound2} implies that 
\begin{equation}
\label{ukBound}
\norm{u_k}_{L^2(B_{R})} \le \abs{B_{R}}^{\frac 1 2} \exp\brac{C_0 \pr{\tfrac{2 }{ \la}}^{2} R_k^2}
\le \exp\brac{C_0 \pr{\tfrac{4 }{ \la}}^{2} R_k^2},
\end{equation}
where we have used \eqref{R1Cond*} and that $R_{k-1} \ge \bar{R}_1$.
On the other hand, the inductive hypothesis shows that 
\begin{align*}
\norm{u_k}_{L^2\pr{B_{\rho + \la^{-1/2}}}} 
&\ge {\la^{n/2}} \norm{u}_{L^2\pr{B_{1}(x_{k-1})}} 
\ge {\la^{n/2}} \exp\brac{- \pr{e^{\tilde c_1 R_1} +1} R_{k-1}^{2+2 \de}} \\
&= {\la^{n/2}} \exp\brac{- \pr{e^{\tilde c_1 R_1} +1} R_{k}^{2}}.
\end{align*}
An application of Proposition \ref{threeBallsA} with $r_1 = \la^{\frac 1 2}$, $r_2= \rho + \la^{-\frac 1 2}$,  $\si r_2 = \frac 1 2 \pr{R + \rho}$, and $r_3 =R$ shows that
\begin{align*}
&\exp\brac{- \pr{e^{\tilde c_1 R_1} +1} R_k^{2}}
\le {\la^{-n/2}} \norm{u_k}_{L^2\pr{B_{r_2}}} \\
&\le \frac{e^{{\bar{C}} \eta R_{k-1}^{\de - \eps}}}{\la^{4 + {\frac n 2}}} \exp\set{ \brac{ \la^{-2} R_k^2  + {\la^{-2 \be \nu}} R_k^{2\be} + C}\brac{C + 2  \log \pr{\frac{2R}{R + \rho}} + \log\pr{\frac{\si^2}{\si^2 -1}}} }  \\
&\times \norm{u_k}_{L^2\pr{B_{\la^{1/2}}}}^{\kappa_k}  \exp\brac{C_0 \pr{\tfrac{4 }{ \la}}^{2} R_k^2},
\end{align*}
where
$\disp \frac 1 {\kappa_k} = 1 + 3e^{{\tilde c_0} \eta R_{k-1}^{\de - \eps} + C_1} \frac{\log \pr{\frac{\rho + R}{2 \sqrt \la}}}{\log \pr{\frac{2R}{R + \rho}}}$
and we have used \eqref{ukBound}.
Since $\disp \frac{2R}{R + \rho} \le 2$, then $\disp \log\pr{\frac{2R}{R + \rho}} \le \log 2 \le 1$. 
As $\disp \si = 1 + \frac{R - \rho - 2 \la^{-\frac 1 2}}{2\pr{\rho + \la^{-\frac 1 2}}}$, then
\begin{align*}
\frac{\si^2}{\si^2 - 1} 
= 1 + \frac{1}{\si^2 - 1}
\le \frac{R - \la^{-\frac 1 2}}{R - \rho - 2 \la^{-\frac 1 2}}
\le \frac{ R_{k-1}^{\de}}{\frac 1 2 \la - 2 R_{k-1}^{-1}}
\le \tfrac 3 \la R_{k-1}^{\de},
\end{align*}
where we have used the second part of \eqref{R1Cond0} and that $R_{k-1} \ge \bar{R}_1$.
Using these bounds in the above expression then shows that
\begin{align*}
\exp\brac{- \pr{e^{\tilde c_1 R_1} +1} R_k^{2}}
&\le \exp\set{ \tilde C R_k^2 \brac{1 + \log\pr{\tfrac 3 \la R_{k-1}^{\de}} }}  \norm{u_k}_{L^2\pr{B_{\la^{1/2}}}}^{\kappa_k}
\end{align*}
where $\tilde C(n, \la, p, C_0) > 0$ is as above.
Since $\disp \frac{\rho + R}{2 \sqrt \la} \le \frac R {\sqrt \la} \le \frac {R_k}\la$ and $\disp \frac {2R}{R + \rho} = 1 + \frac{R - \rho}{R + \rho} \ge 1 + \frac{\la}{4}R_{k-1}^{-\de}$, then the third part of \eqref{R1Cond0} shows that \begin{align*}
\frac{1}{\kappa_k} 
&= 1 + \frac{ 3 e^{{\tilde c_0} \eta R_{k-1}^{\de - \eps} + C_1} \log \pr{\frac{\rho + R}{2 \sqrt \la}} }{\log \pr{\frac{2R}{R + \rho}}}
\le 1 + \frac{ 10 e^{C_1} \log \pr{\la^{-1} R_{k}} }{\log \pr{1 + \frac{\la}{4}R_{k}^{-\frac{\de}{1+\de}}}}
\le R_k^{\frac 3 2 \de},
\end{align*}
where the bound in \eqref{R1Cond2} was used in the last step.
Simplifying the above expression then shows that
\begin{align*}
\norm{u}_{L^2\pr{B_{1}(x_k)}}
\ge \norm{u_k}_{L^2\pr{B_{\la^{1/2}}}}
&\ge \exp\set{-  R_k^{2 + \frac{3}{2}\de} \brac{e^{\tilde c_1 R_1} + \tilde C  \log\pr{\tfrac 3 \la R_{k-1}^{\de}} + \tilde C +1 }} \\
&\ge \exp\brac{- \pr{e^{\tilde c_1 R_1} + 1} \abs{x_k}^{2 + 2\de} },
\end{align*}
where we have applied \eqref{R1Cond3} in the last step.
In particular, the inductive step has been shown.

If $\abs{x_0} \ge \bar R_1$, then there exists $m \in \N$ so that $\bar R_1^{(1 + \de)^{m-1}} \le \abs{x_0} < \bar R_1^{(1 + \de)^{m}}$.
That is, we may find $R_1 \in [\bar R_1, \bar R_1^{1 + \de}]$ so that with $R_2, \ldots, R_m$ defined recursively by $R_k = R_{k-1}^{1+\de}$, we get that $R_m = \abs{x_0}$.
The inductive arguments above show that
\begin{align*}
\norm{u}_{L^2\pr{B_{1}(x_0)}}
&\ge \exp\brac{-\pr{1 + e^{\tilde c_1 R_1}} \abs{x_0}^{2\pr{1 + \de}}}
\ge \exp\brac{-\pr{1 + e^{\tilde c_1 \bar R_1^{1 + \de}}} \abs{x_0}^{2\pr{1 + \de}} },
\end{align*}
and the conclusion described by \eqref{UCatInyResultII} follows.
\end{proof}

%\bibliography{refs}

\begin{thebibliography}{GPSVG18}

\bibitem[ABG81]{ABG81}
W.~O. Amrein, A.-M. Berthier, and V.~Georgescu.
\newblock {$L^{p}$}-inequalities for the {L}aplacian and unique continuation.
\newblock {\em Ann. Inst. Fourier (Grenoble)}, 31(3):vii, 153--168, 1981.

\bibitem[AKS62]{AKS62}
N.~Aronszajn, A.~Krzywicki, and J.~Szarski.
\newblock A unique continuation theorem for exterior differential forms on
  {R}iemannian manifolds.
\newblock {\em Ark. Mat.}, 4:417--453 (1962), 1962.

\bibitem[Alm79]{Alm79}
Frederick~J. Almgren, Jr.
\newblock Dirichlet's problem for multiple valued functions and the regularity
  of mass minimizing integral currents.
\newblock In {\em Minimal submanifolds and geodesics ({P}roc. {J}apan-{U}nited
  {S}tates {S}em., {T}okyo, 1977)}, pages 1--6. North-Holland, Amsterdam-New
  York, 1979.

\bibitem[Alm00]{Alm00}
Frederick~J. Almgren, Jr.
\newblock {\em Almgren's big regularity paper}, volume~1 of {\em World
  Scientific Monograph Series in Mathematics}.
\newblock World Scientific Publishing Co., Inc., River Edge, NJ, 2000.
\newblock $Q$-valued functions minimizing Dirichlet's integral and the
  regularity of area-minimizing rectifiable currents up to codimension 2, With
  a preface by Jean E. Taylor and Vladimir Scheffer.

\bibitem[Aro57]{Aro57}
N.~Aronszajn.
\newblock A unique continuation theorem for solutions of elliptic partial
  differential equations or inequalities of second order.
\newblock {\em J. Math. Pures Appl. (9)}, 36:235--249, 1957.

\bibitem[BG16]{BG16}
Agnid Banerjee and Nicola Garofalo.
\newblock Quantitative uniqueness for elliptic equations at the boundary of
  {$C^{1,{\rm Dini}}$} domains.
\newblock {\em J. Differential Equations}, 261(12):6718--6757, 2016.

\bibitem[BK05]{BK05}
Jean Bourgain and Carlos~E. Kenig.
\newblock On localization in the continuous {A}nderson-{B}ernoulli model in
  higher dimension.
\newblock {\em Invent. Math.}, 161(2):389--426, 2005.

\bibitem[Car39]{Car39}
T.~Carleman.
\newblock Sur un probl\`eme d'unicit\'{e} pur les syst\`emes d'\'{e}quations
  aux d\'{e}riv\'{e}es partielles \`a deux variables ind\'{e}pendantes.
\newblock {\em Ark. Mat. Astr. Fys.}, 26(17):9, 1939.

\bibitem[CK18]{CK18}
Guher Camliyurt and Igor Kukavica.
\newblock Quantitative unique continuation for a parabolic equation.
\newblock {\em Indiana Univ. Math. J.}, 67(2):657--678, 2018.

\bibitem[CL25]{CL25}
Xiujin Chen and Hairong Liu.
\newblock Quantitative uniqueness of solutions to a class of {S}chr\"odinger
  equations with inverse square potentials.
\newblock {\em J. Math. Anal. Appl.}, 543(2):Paper No. 129032, 18, 2025.

\bibitem[CS99]{CS97}
J.~Cruz-Sampedro.
\newblock Unique continuation at infinity of solutions to {S}chr\"{o}dinger
  equations with complex-valued potentials.
\newblock {\em Proc. Edinburgh Math. Soc. (2)}, 42(1):143--153, 1999.

\bibitem[Dav14]{Dav14}
Blair Davey.
\newblock Some quantitative unique continuation results for eigenfunctions of
  the magnetic {S}chr\"odinger operator.
\newblock {\em Comm. Partial Differential Equations}, 39(5):876--945, 2014.

\bibitem[Dav20]{Dav20b}
Blair Davey.
\newblock Quantitative unique continuation for {S}chr\"{o}dinger operators.
\newblock {\em J. Funct. Anal.}, 279(4):108566, 23, 2020.

\bibitem[DZ18]{DZ18}
Blair Davey and Jiuyi Zhu.
\newblock Quantitative uniqueness of solutions to second order elliptic
  equations with singular potentials in two dimensions.
\newblock {\em Calc. Var. Partial Differential Equations}, 57(3):57:92, 2018.

\bibitem[DZ19]{DZ19}
Blair Davey and Jiuyi Zhu.
\newblock Quantitative uniqueness of solutions to second-order elliptic
  equations with singular lower order terms.
\newblock {\em Comm. Partial Differential Equations}, 44(11):1217--1251, 2019.

\bibitem[FK24]{FK24}
N.~D. Filonov and S.~T. Krymskii.
\newblock On the {L}andis conjecture in a cylinder.
\newblock {\em Russ. J. Math. Phys.}, 31(4):645--665, 2024.

\bibitem[GL86]{GL86}
Nicola Garofalo and Fang-Hua Lin.
\newblock Monotonicity properties of variational integrals, {$A_p$} weights and
  unique continuation.
\newblock {\em Indiana Univ. Math. J.}, 35(2):245--268, 1986.

\bibitem[GL87]{GL87}
Nicola Garofalo and Fang-Hua Lin.
\newblock Unique continuation for elliptic operators: a geometric-variational
  approach.
\newblock {\em Comm. Pure Appl. Math.}, 40(3):347--366, 1987.

\bibitem[GPSVG18]{GPSVG18}
Nicola Garofalo, Arshak Petrosyan, and Mariana Smit Vega~Garcia.
\newblock The singular free boundary in the {S}ignorini problem for variable
  coefficients.
\newblock {\em Indiana Univ. Math. J.}, 67(5):1893--1934, 2018.

\bibitem[JK85]{JK85}
David Jerison and Carlos~E. Kenig.
\newblock Unique continuation and absence of positive eigenvalues for
  {S}chr\"odinger operators.
\newblock {\em Ann. of Math. (2)}, 121(3):463--494, 1985.
\newblock With an appendix by E. M. Stein.

\bibitem[Ken06]{Ken06}
Carlos~E. Kenig.
\newblock Some recent quantitative unique continuation theorems.
\newblock In {\em S\'{e}minaire: \'{E}quations aux {D}\'{e}riv\'{e}es
  {P}artielles. 2005--2006}, S\'{e}min. \'{E}qu. D\'{e}riv. Partielles, pages
  Exp. No. XX, 12. \'{E}cole Polytech., Palaiseau, 2006.

\bibitem[KL88]{KL88}
V.~A. Kondrat’ev and E.~M. Landis.
\newblock Qualitative theory of second-order linear partial differential
  equations.
\newblock In {\em Partial differential equations, 3 ({R}ussian)}, Itogi Nauki i
  Tekhniki, pages 99--215, 220. Akad. Nauk SSSR, Vsesoyuz. Inst. Nauchn. i
  Tekhn. Inform., Moscow, 1988.

\bibitem[KT01]{KT01}
Herbert Koch and Daniel Tataru.
\newblock Carleman estimates and unique continuation for second-order elliptic
  equations with nonsmooth coefficients.
\newblock {\em Comm. Pure Appl. Math.}, 54(3):339--360, 2001.

\bibitem[KT16]{KT16}
Abel Klein and C.~S.~Sidney Tsang.
\newblock Quantitative unique continuation principle for {S}chr\"odinger
  operators with singular potentials.
\newblock {\em Proc. Amer. Math. Soc.}, 144(2):665--679, 2016.

\bibitem[Kuk98]{Kuk98}
Igor Kukavica.
\newblock Quantitative uniqueness for second-order elliptic operators.
\newblock {\em Duke Math. J.}, 91(2):225--240, 1998.

\bibitem[Kuk00]{Kuk00}
Igor Kukavica.
\newblock Quantitative, uniqueness, and vortex degree estimates for solutions
  of the {G}inzburg-{L}andau equation.
\newblock {\em Electron. J. Differential Equations}, pages No. 61, 15, 2000.

\bibitem[LMNN25]{LMNN25}
A.~Logunov, E.~Malinnikova, N.~Nadirashvili, and F.~Nazarov.
\newblock The {L}andis conjecture on exponential decay.
\newblock {\em Invent. Math.}, 241(2):465--508, 2025.

\bibitem[LW14]{LW14}
Ching-Lung Lin and Jenn-Nan Wang.
\newblock Quantitative uniqueness estimates for the general second order
  elliptic equations.
\newblock {\em J. Funct. Anal.}, 266(8):5108--5125, 2014.

\bibitem[Mes92]{M92}
V.~Z. Meshkov.
\newblock On the possible rate of decay at infinity of solutions of second
  order partial differential equations.
\newblock {\em Math USSR SB.}, 72:343--361, 1992.

\bibitem[Ngu10]{Tu10}
Tu~A. Nguyen.
\newblock On a question of {L}andis and {O}leinik.
\newblock {\em Trans. Amer. Math. Soc.}, 362(6):2875--2899, 2010.

\bibitem[Pli63]{Pli63}
A.~Pli\'{s}.
\newblock On non-uniqueness in {C}auchy problem for an elliptic second order
  differential equation.
\newblock {\em Bull. Acad. Polon. Sci. S\'{e}r. Sci. Math. Astronom. Phys.},
  11:95--100, 1963.

\bibitem[SS80]{SS80}
M.~Schechter and B.~Simon.
\newblock Unique continuation for {S}chr\"{o}dinger operators with unbounded
  potentials.
\newblock {\em J. Math. Anal. Appl.}, 77(2):482--492, 1980.

\bibitem[Zhu16]{Zhu16}
Jiuyi Zhu.
\newblock Quantitative uniqueness of elliptic equations.
\newblock {\em Amer. J. Math.}, 138(3):733--762, 2016.

\end{thebibliography}
%\bibliographystyle{alpha}

\end{document}